\documentclass[hidelinks,onefignum,onetabnum]{siamart251216}

\usepackage{lipsum}
\usepackage{epstopdf}
\usepackage[normalem]{ulem}
\usepackage{algorithmic}
\ifpdf
  \DeclareGraphicsExtensions{.eps,.pdf,.png,.jpg}
\else
  \DeclareGraphicsExtensions{.eps}
\fi

\newsiamthm{assumption}{Assumption}
\newsiamremark{remark}{Remark}
\newsiamremark{hypothesis}{Hypothesis}
\crefname{hypothesis}{Hypothesis}{Hypotheses}
\newsiamthm{claim}{Claim}
\newsiamremark{fact}{Fact}
\crefname{fact}{Fact}{Facts}
\usepackage{amsmath,amssymb,amsfonts,mathrsfs}
\usepackage{graphicx}
\usepackage{textcomp}
\usepackage{xcolor}
\usepackage{enumerate}
\usepackage{float}
\usepackage{tikz}
\usetikzlibrary{arrows.meta, positioning, fit, decorations.markings,calc}

\allowdisplaybreaks[4]

\tikzset{
    degil/.style={
        postaction={
            decorate,
            decoration={
                markings,
                mark=at position 0.5 with {
                    \draw[-] (-2pt,-2pt) -- (2pt,2pt);
                }
            }
        }
    }
}

\makeatletter
\DeclareOldFontCommand{\rm}{\normalfont\rmfamily}{\mathrm}
\DeclareOldFontCommand{\sf}{\normalfont\sffamily}{\mathsf}
\DeclareOldFontCommand{\tt}{\normalfont\ttfamily}{\mathtt}
\DeclareOldFontCommand{\bf}{\normalfont\bfseries}{\mathbf}
\DeclareOldFontCommand{\it}{\normalfont\itshape}{\mathit}
\DeclareOldFontCommand{\sl}{\normalfont\slshape}{\@nomath\sl}
\DeclareOldFontCommand{\sc}{\normalfont\scshape}{\@nomath\sc}
\makeatother

\usepackage{csquotes}  

\usepackage{physics}   

\newcommand \N   {\mathbb{N}}
\newcommand \R   {\mathbb{R}}

\newcommand \C   {\mathbb{C}}

\newcommand{\vertiii}[1]{{\left\vert\kern-0.25ex\left\vert\kern-0.25ex\left\vert #1 
    \right\vert\kern-0.25ex\right\vert\kern-0.25ex\right\vert}}

\newcommand \qiq   {\quad\Iff\quad}
\newcommand \srs   {\ \ \Rightarrow\ \ }

\newcommand \Iff   {\Leftrightarrow}

\newcommand{\normt}[1]{{\left\vert\kern-0.25ex\left\vert\kern-0.25ex\left\vert #1 
		\right\vert\kern-0.25ex\right\vert\kern-0.25ex\right\vert}}

\renewcommand{\ker}{{\rm Ker}\,}

\newif\ifMath					
\newif\ifEngi					

\newif\ifDFGtext					 

\newif\ifAndo              
													
\newif\ifExercises					
\newif\ifSolutions          
\newif\ifGerman							
\newif\ifEnglish						

\newif\ifnothabil						

\newif\ifFuture							

\newif\ifConf                    
\newif\ifJournal								 

\newif\ifNOTFORBOOK
\newif\ifFullVersion
\newif\ifExludedDueToSpaceReasons

\usepackage{xifthen}

\newcommand{\einsnorm}[2]{\ensuremath{
    \!\!\;\!\!\!\;
    \left\bracevert\!\!\!\!\!\left\bracevert
    \!
		\ifthenelse{\isempty{#2}}{#1}{#1(#2)}
    \!
      \right\bracevert\!\!\!\!\!\right\bracevert
    \!\!\;\!\!\!\;
  }}

\usepackage{xcolor}
\definecolor{blond}{rgb}{0.98, 0.94, 0.75}
	
\newlength\mytemplen
\newsavebox\mytempbox

\makeatletter
\newcommand\mybluebox{%
    \@ifnextchar[
       {\@mybluebox}%
       {\@mybluebox[0pt]}}

\def\@mybluebox[#1]{%
    \@ifnextchar[
       {\@@mybluebox[#1]}%
       {\@@mybluebox[#1][0pt]}}

\def\@@mybluebox[#1][#2]#3{
    \sbox\mytempbox{#3}%
    \mytemplen\ht\mytempbox
    \advance\mytemplen #1\relax
    \ht\mytempbox\mytemplen
    \mytemplen\dp\mytempbox
    \advance\mytemplen #2\relax
    \dp\mytempbox\mytemplen
    \colorbox{blond}{\hspace{1em}\usebox{\mytempbox}\hspace{1em}}}

\makeatother

\makeatletter
\let\origd=\d
\renewcommand*\d{
  \relax\ifmmode
    \mathrm{d}%
  \else
    \expandafter\origd
  \fi
}\makeatother

\usepackage{mathtools}

\makeatletter 
\newcommand{\pushright}[1]{\ifmeasuring@#1\else\omit\hfill$\displaystyle#1$\fi\ignorespaces}
\newcommand{\pushleft}[1]{\ifmeasuring@#1\else\omit$\displaystyle#1$\hfill\fi\ignorespaces}
\makeatother 

\newcounter{syscounter}

\newcounter{WPcounter}
\newcounter{PRcounter}

\headers{From detectability to exponential OSS and back}{Q. Chen, A. Mironchenko, and F. Wirth}

\title{From detectability of abstract linear systems to exponential output-to-state stability and back\thanks{Submitted to the editors DATE.
\funding{This work was partially supported by German Research Foundation (DFG), grants MI 1886/2-2 (project Nr. 415101813) and MI 1886/3-1 (project Nr. 540580186).}}}

\author{
Qiaoling Chen\thanks{Faculty of Computer Science and Mathematics, University of Passau, 94032 Passau, Germany
(\email{qiaoling.chen@uni-passau.de}).}
\and
Andrii Mironchenko\thanks{Department of Mathematics, University of Bayreuth, 95447 Bayreuth, Germany
(\email{andrii.mironchenko@uni-bayreuth.de}).}
\and
Fabian Wirth\thanks{Faculty of Computer Science and Mathematics, University of Passau, 94032 Passau, Germany
(\email{fabian.lastname@uni-passau.de}).}
}

\usepackage{amsopn}

\ifpdf
\hypersetup{
  pdftitle={From detectability of abstract linear systems to exponential output-to-state stability and back},
  pdfauthor={Q. Chen, A. Mironchenko, and F. Wirth}
}
\fi

\begin{document}

\maketitle

\begin{abstract}
We study exponential output-to-state stability (eOSS) of linear infinite-dimensional systems in Banach spaces with bounded output operators. 
It is shown that eOSS is equivalent to the existence of a coercive eOSS Lyapunov function in implication form. 
Also exponential detectability guarantees the existence of a coercive eOSS Lyapunov function in dissipation form and therefore eOSS. Counterexamples demonstrate that the converse implications fail in general: exponential zero-detectability does not imply eOSS, eOSS does not imply the existence of an eOSS Lyapunov function in dissipation form, which, in turn, does not imply exponential detectability.  If the unstable subspace is finite-dimensional, zero-detectability implies exponential detectability and yields equivalent characterizations of eOSS. The results are illustrated with a parabolic equation.

\end{abstract}



\section{Introduction}

Detectability is a fundamental structural property of control systems.  For linear systems, it was first introduced in \cite{wonham1968matrix} as the dual of stabilizability, and soon it became evident that detectability is central in observer design and stabilization by output feedback  \cite{trentelman2001control,Curtain2020}. It also appears in filtering and linear state estimation, especially in  Kalman filters and Riccati-based methods \cite{anderson1979optimal,lancaster1995algebraic,kailath2000linear}. 

For finite-dimensional linear systems, various equivalent ways may be chosen to define detectability. For instance, a spectral definition (``observability of (at least) the unstable modes", \cite{wonham1968matrix}), or the property that a zero output implies asymptotic convergence of the associated state to zero, \cite[p.~317]{sontag2013mathematical}.

The central desired consequence is the existence of an exponentially stabilizing output injection, from which a Luenberger observer with uniformly exponentially stable error dynamics may be constructed.   For infinite-dimensional linear systems, the existence of exponentially stabilizing output injections is now frequently called ``exponential detectability", \cite[Section~8.1]{Curtain2020}, \cite[Section~3.4]{morris2020controller}. The spectral concept used by Wonham, on the other hand, is frequently not sufficient, while the definition using vanishing outputs is again too weak. A further detectability notion of interest is $L^2$-detectability, see e.g. \cite[Definition 11]{koshkin2016positive}, which requires that trajectories corresponding to square integrable output trajectories are in $L^2$ as well.

Detectability derives of course from the notion of observability, and it is
a rule of thumb that detectability is observability on the unstable
part of the dynamics. In the infinite-dimensional case, it is a classical
result that there is a hierarchy of observability concepts ranging from exact observability to
approximate observability. Substantial work has been undertaken to
fully understand the conditions that separate the notions, \cite{haak2012exact,chen2019infinite}. Spectral
conditions for observability and detectability are discussed in \cite{el2011spectral}. A
generalized Hautus condition for exact observability was introduced in
\cite{russell1994general} and counterexamples to the expectation that this condition is necessary and
sufficient were presented in \cite{jacob2004counterexamples}. On the other hand, several cases were identified in which the generalized Hautus condition does in fact guarantee exact observability; see the discussion in \cite{el2011spectral}. 

The relation of exact observability on a finite interval and exponential detectability can be addressed by duality arguments. It is a consequence of the results of \cite{datko1971linear} that in Hilbert spaces exact controllability of linear systems with bounded input operators implies uniform exponential stabilizability. In the general Banach space case, this is false. More precisely, it is shown in \cite{przyluski1988controllability} that a Banach space $X$ has the property that every exactly controllable system with bounded input operator is exponentially stabilizable if and only if $X$ is isomorphic to a Hilbert space. By duality, this shows that in the general Banach space setting exact observability does not imply exponential detectability. We give a concrete example to this effect in Proposition~\ref{prop:not-dete}. In Hilbert spaces, \cite[Example 8.1.2]{Curtain2020} shows that approximate controllability does not imply exponential stabilizability, and again a duality argument shows that there is no implication between approximate observability and exponential detectability.

As a consequence, there are distinct notions of detectability, and a hierarchy of related concepts may be defined.  Spectral conditions are again relevant for specific system classes.  In the particular case of linear delay systems, spectral detectability has been used in the context of output feedback design, \cite{fiagbedzi1990output}. In general, however, exponential detectability is not characterized by them. We will also consider exponential zero detectability, by which we mean that trajectories corresponding to zero outputs satisfy uniform bounds of exponential decay; see Definition~\ref{def:exp-zero-detect}.

There is a close relation of the concept of exponential detectability to Lyapunov theory, see \cite{koshkin2016positive} for a discussion of the related theory in Banach spaces, and  \cite{zabczyk1975remarks} for the relations of exponential detectability and the solutions to algebraic Riccati equations in Hilbert spaces. As we will see in Proposition~\ref{prop:coer eOSS LF}, exponential detectability yields a natural OSS Lyapunov function in dissipation form in terms of an equivalent norm in which the closed-loop system is a contraction. 

All of this raises the question whether a more refined analysis of distinct concepts of detectability is possible.
In this paper, we study output-to-state stability (OSS) for linear systems with bounded output operator and show that this defines a system property between exponential detectability and exponential zero detectability. In addition, the characterization of OSS in terms of Lyapunov functions is investigated.

Output-to-state stability imposes bounds on the evolution of trajectories in terms of an asymptotically decaying influence of the initial condition together with a bound derived from the observed output. 
The concept was introduced in \cite{sontag1996} with the aim to generalize notions of detectability to nonlinear finite-dimensional systems. For finite-dimensional linear systems it is known to be equivalent to detectability, \cite{sontag1997output,krichman2001input}.
In addition, a general characterization of OSS in terms of OSS Lyapunov functions both in implication and in dissipation form was obtained in the finite-dimensional case, \cite{sontag1997output}.

We show in this paper that the equivalence of all these concepts falls apart already for linear infinite-dimensional systems with bounded output operator. Exponential detectability remains the strongest property. The two potentially possible ways of describing OSS via Lyapunov functions turn out to be nonequivalent. Only the implication formulation of OSS Lyapunov functions is in fact equivalent to OSS. Examples are provided that exhibit the distinction between the different concepts.

OSS theory and its extensions to systems with inputs and outputs, such as (incremental) input-output-to-state stability, have a broad range of applications. Incremental IOSS is a necessary condition for the existence of robust full-state observers \cite{sontag1997output} and other types of state estimators, including the extended Kalman filter, full information estimation, and moving horizon estimation \cite{allan2021nonlinear}. Beyond observer design, incremental IOSS plays a central role in optimization-based state estimation \cite{RaM09}. Furthermore, for optimal control problems with nonnegative stage costs, the existence of an IOSS-like Lyapunov function on bounded sets implies strict dissipativity for cost functions minimizing a function of the output norm \cite{HoG19}. Strict dissipativity, in turn, yields the turnpike property, which asserts that optimal trajectories corresponding to different time horizons exhibit a common long-term behavior \cite{GrM16,MGA15,TrZ15}. This property is fundamental for the analysis and performance of model predictive control schemes \cite{GrP17,RMD19,AlZ00}.

While the infinite-dimensional OSS theory is making its first steps, its older sibling - the input-to-state stability (ISS) framework - is rather well-developed in the infinite-dimensional setting.

ISS provides a systematic language for stability estimates with respect to external signals \cite{sontag1989smooth}. ISS quantifies the effect of inputs on the state, while input-to-output stability (IOS) describes the corresponding effect on measured outputs \cite{sontag1999notions}. Input-output-to-state stability (IOSS) combines input and output information in an estimate of the state and is closely related to robust detectability \cite{krichman2001input}.  Detectability-type assumptions also arise in recent output-feedback stabilization results for nonlinear finite and infinite-dimensional systems \cite{preuster2026stabilization}. Incremental IOSS strengthens this viewpoint by comparing pairs of trajectories and is connected with robust state estimation and observer design \cite{allan2021nonlinear}.

For nonlinear evolution equations on Banach spaces, ISS characterizations and Lyapunov methods have been studied in \cite{jayawardhana2008infinite,Karafyllis:2016a,Karafyllis:2018iss,
mironchenko2017characterizations,mironchenko2018lyapunov,
mironchenko2020input,Chaillet2023}. For linear infinite-dimensional systems, semigroup methods and admissibility theory are central, especially when input operators are unbounded \cite{Jacob:2018_SIAM,Schwenninger2020}. For ODE systems, ISS can be characterized through IOS and IOSS, and analogous connections have recently been investigated for infinite-dimensional systems \cite{Bachmann2024}.


Compared with the relatively mature development of ISS for infinite-dimensional systems, OSS theory is still at a much earlier stage. This gap is already visible for time-delay systems, a canonical class of infinite-dimensional systems, where the extension of OSS and IOSS theory was pointed out as missing in the survey by \cite[Section 8.10]{Chaillet2023}. A first step toward an OSS theory for evolution equations on Banach spaces was made in \cite{chen2025lyapunov}, where direct OSS Lyapunov theorems, the vanishing output vanishing state property, and uniform global asymptotic stability modulo output were studied. However, converse Lyapunov results and the precise relations between output-to-state estimates, detectability properties, and Lyapunov characterizations were not addressed there.

\textbf{Contribution.} 
The present paper develops a systematic OSS theory for linear infinite-dimensional
systems with bounded output operators on Banach spaces. Our first result concerns the relationship between OSS and its exponential counterpart, exponential output-to-state stability (eOSS). Although asymptotic and
uniform exponential stability are distinct for general infinite-dimensional linear
systems, the linear structure can be used to derive
exponential estimates with linear gain from a simple OSS condition. Consequently, OSS, OSS with a linear
gain, eOSS, and eOSS with linear gain are equivalent for the systems studied in this paper.

Next, we show that \emph{eOSS is equivalent to the existence of a coercive eOSS Lyapunov function in implication form}, which provides a powerful tool for the analysis of OSS.
At the same time, we show by counterexample that \emph{there are eOSS systems for which there is no coercive OSS Lyapunov function in dissipation form}. 
The fact that the existence of an OSS Lyapunov function in dissipation form is more demanding than the existence of an OSS Lyapunov function in implication form -- even for linear systems with bounded output operators -- is rather surprising, and suggests that implication-form Lyapunov functions may be, after all, more suitable for eOSS analysis.

Next, we relate exponential OSS to the classical detectability concepts.
 Exponential detectability is shown to
imply the existence of a coercive eOSS Lyapunov function in dissipation
form with linear gain. The construction is based on an equivalent norm generated by an
exponentially stable output-injected semigroup. Consequently, exponential
detectability implies eOSS.

However, the converse implications fail.
Counterexamples are provided showing that \emph{exponential zero-detectability does not imply eOSS, that eOSS does not imply the existence of a coercive eOSS Lyapunov function in dissipation form, and that the existence of
such an eOSS Lyapunov function does not imply exponential detectability.}  
We summarize these relations in Figure~\ref{fig:Result}.

Finally, we identify an additional structural condition under which the finite-dimensional picture is recovered. If the unstable subspace is finite-dimensional, in the
sense of an invariant decomposition with an exponentially stable complement, then exponential zero-detectability implies exponential detectability. Under this assumption, all different detectability properties studied in this paper become equivalent.

Our findings show that there is a whole world between the classical concepts of exponential detectability and exponential zero-detectability - centered around the notion of output-to-state stability. As this concept can be naturally formulated for nonlinear systems and has an equivalent characterization in terms of OSS Lyapunov functions in implication form, our results suggest that OSS can serve as a canonical detectability property for infinite-dimensional systems, as it is in the setting of nonlinear ODEs \cite{sontag2008input}. 

A six-page conference version of this paper has been presented at the MTNS 2026 conference \cite{chen2026exponential}.

\textbf{Outlook.}
This work can serve as a firm basis for the study of OSS and detectability for infinite-dimensional systems. 
A natural direction for future work is the study of IOSS for linear infinite-dimensional systems, including Lyapunov characterizations and
connections with detectability, especially in the presence of unbounded input or output operators.
Some of the counterexamples presented in this work are given for strongly continuous semigroups in non-reflexive Banach spaces.
The question of whether some of the implications can be valid in the case of reflexive Banach spaces, Hilbert spaces, or if the semigroup is analytic, remains open as well. 
One of the main objectives is the development of the nonlinear infinite-dimensional OSS theory.

\textbf{Structure.}
The remainder of the paper is organized as follows.
Section~\ref{sec:prelim} introduces the system class, OSS/eOSS notions, and
Lyapunov functions, and proves the equivalence of the OSS and eOSS variants.
In Section~\ref{sec:Converse eOSS Lyapunov theorems} we prove the converse eOSS Lyapunov theorem (in implication form).
Section~\ref{sec: exp detectability} relates eOSS Lyapunov functions to
exponential detectability.
Section~\ref{sec:counterexam} provides counterexamples to the converse
implications.
Section~\ref{sec: decomposition} gives an observer interpretation and derives equivalent characterizations of eOSS
when the unstable subspace is finite-dimensional.
Section~\ref{sec: example} applies the results to a parabolic equation.
Section~\ref{concl} concludes the paper and discusses future research
directions.

\textbf{Notation.}  Let $\mathbb{R}$, $\mathbb{C}$, $\mathbb{N}$ be the sets of real numbers, complex numbers, and positive integers, respectively. We also set $\mathbb{R}_+:=[0,\infty)$. Given topological spaces $X, Y$, the  space of continuous functions from $X$ to $Y$ is denoted by
$C(X, Y)$ and 
$C^2(0,1)$ denotes the space of real-valued functions that are twice continuously differentiable on the interval $(0,1)$. Further, $L^\infty_{\, \mathrm{loc}}([0,\infty),Y)$ denotes the space of locally essentially bounded, strongly measurable functions from $[0,\infty)$ to $Y$.

Let $X$ and $Y$ be Banach spaces. We denote by 
$\mathcal{L}(X,Y)$ the space of bounded linear operators from $X$ to $Y$, endowed with the operator norm $\|\cdot\|_{\mathcal{L}(X,Y)}$, and  write $\mathcal L(X):=\mathcal L(X,X)$. For the operator norm on $\mathcal{L}(X)$, we simply write $\|\cdot\|.$
 A function $V:X\to\mathbb{R}_+$ is called $p$-absolutely homogeneous, $p>0$, if
$V(ax)=|a|^pV(x)$
for all $a\in\mathbb{R}$ and $x\in X$. When spectral notions are used for operators on real Banach spaces, the corresponding complexifications are understood.
We use the following notation for the classes of  comparison functions:
\begin{align*}
\mathcal{P}&:=\{\gamma:\mathbb{R}_+\rightarrow\mathbb{R}_+ \mid \gamma\ \text{is continuous with}
\ \gamma(0)=0\ \text{and}\ \gamma(r)>0\
\text{for}
\ r>0\},\\
\mathcal{K}&:=\{\gamma\in\mathcal{P} \mid \gamma\ \text{is strictly increasing}\},\\
\mathcal{K}_\infty&:=\{\gamma\in \mathcal{K}\mid \gamma\ \text{is unbounded}\},\\
 \mathcal{L}&:=\{\gamma:\mathbb{R}_+\rightarrow\mathbb{R}_+\mid \gamma\ \text{is continuous, strictly  decreasing with}\ \lim_{r\rightarrow\infty}\gamma(r)=0 \},\\
\mathcal{KL}&:=\{\beta: \mathbb{R}_+\times \mathbb{R}_+\rightarrow\mathbb{R}_+\mid \beta\ \text{is continuous}, \beta(\cdot,t)\in\mathcal{K},
 \forall t\geq0, \beta(r,\cdot)\in\mathcal{L}, \forall r>0\}.
\end{align*}

\section{Problem setting}\label{sec:prelim} 

Let $X$ and $Y$ be Banach spaces (the state space and the space of output values, respectively). Consider the system
\begin{equation}\label{eq:linear-system}
\dot{x}(t) = Ax(t), \quad y(t) = Cx(t), 
\end{equation}
where $A:D(A)\subset X\to X$ generates a $C_0$-semigroup $T(\cdot)$ on $X$, and $C\in\mathcal{L}(X,Y)$ is the output operator.

We denote by $\phi(t,x)$ the mild solution of system \eqref{eq:linear-system} at time $t\ge 0$
corresponding to the initial condition $x\in X$, i.e.,
$\phi(t,x)=T(t)x$.
For each 
$x \in X$, the function $\phi(\cdot, x)$ belongs to the vector space $C(\mathbb R_+, X)$ of continuous $X$-valued functions on $\mathbb R_+$.

We choose the space of output trajectories  as the vector space
\[
\mathcal Y := C(\mathbb R_+,Y).
\]
Slightly abusing notation, we denote by $y(\cdot,x):=CT(\cdot)x$ the output trajectory corresponding to the initial condition $x \in X$.

We first recall the notion of output-to-state stability as introduced in \cite{sontag1996}. 

 \begin{definition}\label{def:OSS} 
 System  \eqref{eq:linear-system} is   called \emph{output-to-state stable (OSS)}, if there exist $\beta\in\mathcal{KL}$ and $\gamma\in\mathcal{K}$ such that for all $ x\in X$ and all $t\geq0$ it holds that
 \begin{align}
\label{abOSS}
 \|\phi(t,x)\|_{X}\leq \beta(\|x\|_{X},t)+ \gamma\left(\max\limits_{s\in[0,t]}\|y(s,x)\|_Y\right).
 \end{align}
If in addition $\gamma$ can be chosen as a linear function, we call \eqref{eq:linear-system} \emph{OSS}  \emph{with linear gain}.
 \end{definition}

 We now define the exponential version of OSS.

 \begin{definition}\label{def:eOSS}   
System  \eqref{eq:linear-system} is called \emph{exponentially output-to-state stable \linebreak(eOSS)}, if there exist $M,\mu>0$ and $\gamma\in\mathcal{K}$ such that for all $ x\in X$ and all $t\geq0$ it holds that
\begin{align}
\label{eOSS}
\|\phi(t,x)\|_{X}\leq Me^{-\mu t}\|x\|_{X}+ \gamma\left(\max\limits_{s\in[0,t]}\|y(s,x)\|_Y\right).
\end{align}

\medskip
If in addition $\gamma$ can be chosen as a linear function, then we call \eqref{eq:linear-system} \emph{eOSS} \emph{with linear gain}.

If in \eqref{eOSS} we can choose $\gamma=0$, then the semigroup $T(\cdot)$ is
called \emph{(uniformly) exponentially stable (UES)}.

\end{definition}

If the inequality \eqref{eOSS} holds, then bounded outputs imply bounded trajectories. Furthermore, vanishing outputs imply vanishing states; see \cite{chen2025lyapunov}.

For general nonlinear systems, OSS, eOSS, both with or without linear gain  need not be equivalent. As we see now, they are equivalent for the linear system \eqref{eq:linear-system}.

\begin{proposition}
\label{prop:equi OSS eOSS LG}
The following statements are equivalent:
 \begin{enumerate}
 \item[(i)] System \eqref{eq:linear-system} is OSS.
 \item[(ii)] System \eqref{eq:linear-system} is OSS with linear gain.
 \item[(iii)]  System \eqref{eq:linear-system} is eOSS.
 \item[(iv)]  System \eqref{eq:linear-system} is eOSS with linear gain.
    \end{enumerate}
\end{proposition}

\begin{proof}
The implications (iv) $\Rightarrow$ (ii) $\Rightarrow$ (i) and (iv) $\Rightarrow$ (iii) $\Rightarrow$ (i) are immediate. Hence, it suffices to prove (i) $\Rightarrow$ (iv). 

(i) $\Rightarrow$ (iv). Assume that system \eqref{eq:linear-system} is OSS with $\beta\in\mathcal{KL}$ and $\gamma\in\mathcal{K}$ such that \eqref{abOSS} holds. 
The case $x=0$ is trivial, so we only consider $x\neq0$.

Since $\beta\in\mathcal{KL}$, there exists $t_1>0$ such that $\beta(1,t_1)<1$. Choose $\varepsilon>0$ such that 
\[
\delta:=\beta(1,t_1)+\gamma(\varepsilon)<1.
\]
Recall the notation $y(s,x):=CT(s)x$, $s\ge 0$, $x \in X$. 
We first show that there exists $K=K(\varepsilon,t_1)\ge0$ such that
\begin{align}\label{eq:key-step}
\|T(t_1)x\|_X\le \delta\|x\|_X+K \max\limits_{s\in[0,t_1]}\|y(s,x)\|_Y,\quad x\in X.
\end{align}

 To this end we consider two cases.  On the one hand, if $\max\limits_{s\in[0,t_1]}\|y(s,x)\|_Y\le \varepsilon\|x\|_X$, then applying \eqref{abOSS} to $\|x\|_X^{-1}x$ at time $t_1$ yields by linearity
\begin{align*}
\bigl\|T(t_1)\tfrac{x}{\|x\|_X}\bigr\|_X
\le \beta(1,t_1)+\gamma\bigl(\|x\|_X^{-1}\max\limits_{s\in[0,t_1]}\|y(s,x)\|_Y\bigr)
\le \beta(1,t_1)+\gamma(\varepsilon)
= \delta.
\end{align*}
Thus, $\|T(t_1)x\|_X\le \delta\|x\|_X$. On the other hand, if $\max\limits_{s\in[0,t_1]}\|y(s,x)\|_Y> \varepsilon\|x\|_X$, then  
\begin{align*}
\|T(t_1)x\|_X
\le \|T(t_1)\|\|x\|_X
\le \frac{\|T(t_1)\|}{\varepsilon}\max\limits_{s\in[0,t_1]}\|y(s,x)\|_Y.
\end{align*}

Thus \eqref{eq:key-step} holds with $K=\frac{\|T(t_1)\|}{\varepsilon}$.

Fix $n\in \N$. To iterate \eqref{eq:key-step}, we apply it to $T(jt_1)x$, $j=0,\ldots,n-1$. 
By the semigroup property, we have
\begin{align*}
\max\limits_{s\in[0,t_1]}\|y(s,T(jt_1)x)\|_Y
=
\max_{0\le s\le t_1}\|CT(s)T(jt_1)x\|_Y
&=
\max_{jt_1\le s\le (j+1)t_1}\|CT(s)x\|_Y\\
&\le
\max\limits_{s\in[0,nt_1]}\|y(s,x)\|_Y.
\end{align*}
Therefore, iterating \eqref{eq:key-step} and using the semigroup property, we obtain
\begin{align*}
\|T(nt_1)x\|_X
\le\;& \delta^n\|x\|_X+(1+\delta+\cdots+\delta^{n-1})K \max\limits_{s\in[0,nt_1]}\|y(s,x)\|_Y\\
\le\;& \delta^n\|x\|_X+\frac{K}{1-\delta}\max\limits_{s\in[0,nt_1]}\|y(s,x)\|_Y.
\end{align*}

Let $t\ge0$ and write $t=nt_1+\tau$ with $n\in \N \cup\{0\}$ and $\tau\in[0,t_1)$; for $n=0$ the preceding bound holds trivially. Then
\begin{align*}
\|T(t)x\|_X
=\; \|T(\tau)T(nt_1)x\|_X
\le\;& \|T(\tau)\|\|T(nt_1)x\|_X\\
\le\;& \|T(\tau)\|\delta^n\|x\|_X
+\|T(\tau)\|\frac{K}{1-\delta}\max\limits_{s\in[0,nt_1]}\|y(s,x)\|_Y.
\end{align*}

Since $T(\cdot)$ is a $C_0$-semigroup, it is uniformly bounded on the compact interval $[0,t_1]$. 
Moreover, since $\max\limits_{s\in[0,nt_1]}\|y(s,x)\|_Y\le \max\limits_{s\in[0,t]}\|y(s,x)\|_Y$ and $\tau\in[0,t_1]$, we obtain
\begin{align*}
\|T(t)x\|_X
\le\;& \sup_{s\in[0,t_1]}\|T(s)\|\delta^n\|x\|_X+\frac{K}{1-\delta}\sup_{s\in[0,t_1]}\|T(s)\| \max\limits_{s\in[0,t]}\|y(s,x)\|_Y.
\end{align*}

Noting that $t<(n+1)t_1$, we have $n>\frac{t}{t_1}-1$. Since $0<\delta<1$, we get
\begin{align*} \delta^n=e^{n \ln \delta}\le e^{(\frac{t}{t_1}-1) \ln \delta} \le \delta^{-1}e^{\frac{\ln\delta}{t_1}t}. \end{align*} Therefore, 
\begin{align*}
\|T(t)x\|_X
\le\;&\delta^{-1} \sup_{s\in[0,t_1]}\|T(s)\| e^{\frac{\ln\delta}{t_1}t}\|x\|_X+\frac{K}{1-\delta}\sup_{s\in[0,t_1]}\|T(s)\| \max\limits_{s\in[0,t]}\|y(s,x)\|_Y.
\end{align*}

Thus, the system \eqref{eq:linear-system} is eOSS with linear gain.
\end{proof}

\subsection{Lyapunov functions}

This subsection introduces the basic notions needed for the Lyapunov analysis of eOSS. First, the Dini derivative is recalled. Then eOSS Lyapunov functions in dissipation and implication form are defined, and the relationship between these two formulations is investigated.

To describe growth properties of functions that are not necessarily differentiable, we use Dini derivatives as a substitute for classical derivatives.

For a continuous function $q:\mathbb{R}\to\mathbb{R}$, define the (upper right-hand) \emph{Dini derivative} of $q$ at $t$ by
\begin{align*}
D^+q(t)
=\mathop{\overline{\lim}}\limits_{\tau\rightarrow+0}
\frac{q(t+\tau)-q(t)}{\tau}.
\end{align*}
The above limsup is taken in the extended real line, allowing $D^+q(t)$ to attain the values $-\infty$ and $+\infty.$
In what follows, we apply this notion to the  function
$t\mapsto V(\phi(t,x))$, whenever this function is continuous.
In addition, define the \emph{Lie derivative} of $V$ at $x \in X$ by
\[
\dot{V}(x):=D^+V(\phi(t,x))|_{t=0}=
\mathop{\overline{\lim}}\limits_{\tau\rightarrow+0}
\frac{V(\phi(\tau,x))-V(x)}{\tau},
\]
provided that $t\mapsto V(\phi(t,x))$ is continuous at $t=0$.

Lyapunov functions serve as a useful tool for studying the eOSS property.  Their evolution along trajectories is described by the Dini derivative. 
We begin by defining eOSS Lyapunov functions in dissipation and implication forms. The former follows the Lyapunov characterization of OSS in \cite{sontag1996}, while the latter is  adapted from
the implication form of the ISS Lyapunov characterization in \cite{sontag1995characterizations}.

\begin{definition}
\label{def:eOSS LF-dissipative}
A   function $V: X\rightarrow \mathbb{R}_{+}$ is called a  \emph{coercive eOSS Lyapunov function in dissipation form} for  system \eqref{eq:linear-system} if:
 \begin{enumerate}
\item[(i)] There exist $k_1,k_2,p>0$ such that for every $x\in X$:
\begin{align}\label{eq:eoss LF1}
k_1\|x\|_{X}^p\leq V(x)\leq k_2\|x\|_{X}^p.
\end{align}

\item[(ii)] For every $x\in X$,  the map $t\mapsto V(\phi(t,x))$, $t\geq0$, is continuous.
\item[(iii)] There exist $\alpha>0$ and $\sigma\in\mathcal{K}_{\infty}$ such that for every $x\in X$,
the following dissipation inequality holds with $p$ as in (i)
\begin{align}
\label{eoss LF2}
\dot{V}(x)\leq-\alpha\|x\|_X^p+\sigma(\|Cx\|_{Y}).
\end{align}

\end{enumerate}

If  \eqref {eoss LF2} holds with $\sigma\equiv0$, then $V$ is called a \emph{coercive UES Lyapunov function}.

The function  $V$ is called a  \emph{coercive eOSS Lyapunov function in  implication form} for \eqref{eq:linear-system} if
items (i) and (ii) hold, and (iii) is replaced by
\begin{enumerate}    
\item[(iii')] there exist $\tilde{\alpha}>0$ and $\chi\in\mathcal{K}_{\infty}$ such that for every $x\in {X}$, we have with $p$ as in (i)
\begin{equation}
\label{eq:aboss Implication}
\hspace{-4mm}\|x\|_{X} \geq\chi(\|Cx\|_{Y})
\quad \Rightarrow\quad  \dot{V}(x)\leq-\tilde{\alpha}\|x\|_X^p.
\end{equation}
\end{enumerate}
If the Lyapunov gain $\sigma$ or $\chi$ can be chosen to be linear, then $V$ is said to have a \emph{linear Lyapunov gain}.
\end{definition}

The inequality \eqref{eoss LF2} gives the estimate for the Lie derivative of $V$ on the whole space $X$. In contrast, the implication \eqref{eq:aboss Implication} provides an estimate for $\dot{V}$ only on the region $\{x\in X \mid \|x\|_{X} \ge \chi(\|Cx\|_{Y})\}$. On the complement of this region, i.e., on the set $\{x\in X \mid \|x\|_{X} < \chi(\|Cx\|_{Y})\}$, \eqref{eq:aboss Implication} does not provide any information about the behavior of the derivative.
The next proposition provides an equivalent characterization of coercive eOSS Lyapunov functions in dissipation form in terms of the implication form supplemented by an additional estimate on the complementary region.

\begin{proposition}\label{prop:dissi-impli}
    Consider $V: X\rightarrow \mathbb{R}_{+}$. For system \eqref{eq:linear-system} the following statements are equivalent:
    \begin{enumerate}
        \item [(i)] $V$ is a coercive eOSS Lyapunov function in dissipation form. 
        \item [(ii)]$V$ is a coercive eOSS Lyapunov function in implication form with parameters $p,\tilde{\alpha}>0$ and $\chi\in\mathcal{K}_{\infty}$, and there exists $\eta\in\mathcal{K}_\infty$  such that for every $x\in X$     
\begin{align}\label{eq:addassm}       
\|x\|_{X} < \chi(\|Cx\|_{Y})
\quad \Rightarrow\quad  \dot{V}(x)\leq \eta(\|Cx\|_Y).     
\end{align}
    \end{enumerate}
\end{proposition}

\begin{proof}
 (i) $\Rightarrow$ (ii). Let $V$ be a coercive eOSS Lyapunov function in dissipation form. Then \eqref{eq:eoss LF1} and \eqref{eoss LF2} hold for appropriate $k_1,k_2,p,\alpha>0$ and $\sigma\in \mathcal{K}_\infty$.
Define $\chi(r):=\bigl(\frac{2}{\alpha}\sigma(r)\bigr)^{1/p}, r\in\mathbb{R}_+$. Consider $x\in X$ with $\|x\|_X\ge \chi(\|Cx\|_Y)$.
Then, by the definition of $\chi$,
\begin{equation*}
\sigma(\|Cx\|_Y)\le \frac{\alpha}{2}\|x\|_X^p.
\end{equation*}
Hence, by~\eqref{eoss LF2},
\begin{equation*}
  \|x\|_{X} \ge \chi(\|Cx\|_{Y})
\quad \Rightarrow\quad \dot V(x)
\le -\frac{\alpha}{2}\|x\|_X^p.
\end{equation*}
Therefore, $V$ is a coercive eOSS Lyapunov function in implication form with $\tilde{\alpha}=\frac{\alpha }{2}>0$.
Moreover, by \eqref{eoss LF2}, for every $x\in X$,
\begin{equation*}
\dot{V}(x)\leq-\alpha\|x\|_X^p+\sigma(\|Cx\|_Y)
\leq \sigma(\|Cx\|_Y).
\end{equation*}
Hence,
\begin{equation*}
\|x\|_{X} < \chi(\|Cx\|_{Y})
\quad \Rightarrow\quad  \dot{V}(x)\leq \sigma(\|Cx\|_Y).
\end{equation*}
Thus, the additional condition holds with $\eta=\sigma\in\mathcal{K}_\infty$. This shows (ii).

 (ii) $\Rightarrow$ (i). Let $V$ be a coercive eOSS Lyapunov function in implication form with $p,\tilde{\alpha}>0$ and $\chi\in\mathcal{K}_{\infty}$. Suppose that there exists $\eta\in\mathcal{K}_\infty$ such that the implication in \eqref{eq:addassm} holds.  Define 
\begin{align*} \sigma(r):=\eta(r)+\tilde{\alpha}\chi(r)^p,
\quad r\ge0. \end{align*}
Then $\sigma\in\mathcal{K}_\infty.$ Let $x\in X.$ If $\|x\|_{X} \geq\chi(\|Cx\|_{Y})$,  then, by \eqref{eq:aboss Implication}, 
\begin{align*}
    \dot{V}(x)\leq-\tilde{\alpha}\|x\|_X^p\le 
    -\tilde{\alpha}\|x\|_X^p+\sigma(\|Cx\|_Y).
\end{align*}
If $\|x\|_{X}<\chi(\|Cx\|_{Y})$, then, by the additional assumption \eqref{eq:addassm},
\begin{align}
\label{eq:Making-LF-better}
    \dot{V}(x)\leq \eta(\|Cx\|_Y)
    &= -\tilde{\alpha}\|x\|_X^p+\tilde{\alpha}\|x\|_X^p+\eta(\|Cx\|_Y)\nonumber\\
    &\le  -\tilde{\alpha}\|x\|_X^p +\tilde{\alpha}\chi(\|Cx\|_Y)^p+\eta(\|Cx\|_Y)\nonumber\\
    &=  -\tilde{\alpha}\|x\|_X^p +\sigma(\|Cx\|_Y).
\end{align}
Thus, for every $x\in X$,  $\dot{V}(x)\leq         
-\tilde{\alpha}\|x\|_X^p+\sigma(\|Cx\|_Y)$. 
Therefore, $V$ is a coercive eOSS Lyapunov function in dissipation form.
\end{proof}

 As an immediate consequence of Proposition~\ref{prop:dissi-impli} we note the following.
\begin{corollary}\label{cor:dissi-imp}
    For system \eqref{eq:linear-system},
    every coercive eOSS Lyapunov function in dissipation form is also a coercive eOSS Lyapunov function in implication form.
\end{corollary}

\begin{remark}
Indeed, Proposition~\ref{prop:no eOSS LF} below provides an eOSS system
which admits a coercive eOSS Lyapunov function in implication form, but does
not admit any eOSS Lyapunov function in dissipation form. In view of
Proposition~\ref{prop:dissi-impli}, this shows that the additional assumption \eqref{eq:addassm} cannot be omitted.
\end{remark}

\begin{corollary}\label{cor:di-im-lg}
    For system \eqref{eq:linear-system}, if  $V$ is a coercive eOSS Lyapunov function in dissipation form with linear Lyapunov gain and $p=1$, then $V$ is a coercive eOSS Lyapunov function in implication form with linear Lyapunov gain and $p=1$. 
\end{corollary}

\begin{proof}
 This follows directly from the definition of $\chi$ in the proof of Proposition~\ref{prop:dissi-impli}, in the part (i) $\Rightarrow$ (ii).   
\end{proof}

\begin{remark}   

In view of the computations \eqref{eq:Making-LF-better}, on the complementary region $\|x\|_{X} <\chi(\|Cx\|_{Y})$, we have
\[
\exists \eta\in\mathcal{K}_\infty: \ \dot{V}(x)\leq \eta(\|Cx\|_Y) \qiq \exists \sigma\in\mathcal{K}_\infty:  \ \dot{V}(x)\le -\tilde{\alpha}\|x\|_X^p+\sigma(\|Cx\|_Y).    
\]

This shows that the additional estimate in \eqref{eq:addassm} can be expressed in the same form as the dissipation inequality on the complementary region.
These complementary regions (defined for linear $\chi$ in \eqref{eq:E-and-D-sets}) play an important role in the proof of the converse Lyapunov theorem (Theorem~\ref{thm:convserse-eOSS}).

\end{remark}

\begin{remark}[Non-coercive eOSS Lyapunov functions]
In \cite{mironchenko2017characterizations,JMP20} the concept of a non-coercive ISS Lyapunov function has been proposed.
Here non-coercivity means that $V$ just satisfies 
\begin{equation}
 \label{eq:noncoercive}
     V(0)=0,\quad 0<V(x)\le k_2\|x\|_X^p,\quad x\in X\setminus\{0\}.
 \end{equation}
 It was shown that under certain assumptions, the existence of such a Lyapunov function implies input-to-state stability of a control system with inputs (see \cite{mironchenko2017characterizations,JMP20} for definitions). Similarly, one could introduce a \emph{non-coercive eOSS Lyapunov function}
 as a function $V: X\to \R_+$ for which conditions (ii) and (iii), respectively conditions (ii) and (iii'), of Definition~\ref{def:eOSS LF-dissipative} hold, while condition (i) is replaced by \eqref{eq:noncoercive}.
 
The coercivity of $V$ is not used in the proof of Proposition~\ref{prop:dissi-impli}. Hence, the same argument yields the corresponding equivalence for non-coercive eOSS Lyapunov functions.
Note, however, that right now it is not known whether non-coercive OSS Lyapunov functions can be exploited to infer OSS of dynamical systems.
\end{remark}

\section{Converse eOSS Lyapunov theorems}
\label{sec:Converse eOSS Lyapunov theorems}

In \cite{chen2025lyapunov}, we have shown that the existence of an OSS Lyapunov function implies OSS of a nonlinear system. For linear systems, this result can be used to prove that the existence of a coercive eOSS Lyapunov function implies eOSS. As in many cases, constructing an eOSS Lyapunov function is the most realistic way to prove eOSS, which makes direct Lyapunov results significant for eOSS theory. 
In this section, we show a converse result: For every eOSS system there exists an eOSS Lyapunov function in implication form. As it turns out, the type of Lyapunov function is central at this point as a similar statement is false for Lyapunov functions in dissipation form, as we will see in Section~\ref{sec:counterexam}.
Furthermore, we show that eOSS is equivalent to the following property: there are uniform bounds of exponential decay such that on any interval, on which the state dominates the output, the state satisfies this uniform bound.

\begin{definition}\label{UGESMO}  System \eqref{eq:linear-system} is called \emph{uniformly globally exponentially  stable modulo output (UGESMO)}, if there exist $\rho>0$ and $M,\mu>0 $ such that for all $x\in X$ and all $\tau\geq0$, if
\begin{align*}
\|\phi(t,x)\|_X\geq\rho \|y(t,x)\|_Y \quad\forall t\in[0,\tau], 
\end{align*} 
then 
\begin{equation} 
\label{eq:UGESMO-decay}
\|\phi(t,x)\|_X\leq M e^{-\mu t}\|x\|_X \quad\forall t\in[0,\tau]. 
\end{equation} 
\end{definition}

The nonexponential version of UGESMO was used in \cite{krichman2001input}, and was our primary motivation. 
For infinite-dimensional systems, it was related to OSS in \cite{chen2025lyapunov}.

\begin{remark}
\label{OSS-implication-form}  
Another implication-form formulation of eOSS is as follows:

System \eqref{eq:linear-system} is called \emph{exponentially OSS in implication form}, if there exist $\rho>0$ and $M,\mu>0 $ such that for all $x\in X$ and $t \ge 0$, we have:
\begin{align*}
\|\phi(t,x)\|_X\geq\rho \max_{s\in[0,t]} \|y(s,x)\|_Y \srs  \|\phi(t,x)\|_X\leq M e^{-\mu t}\|x\|_X. \end{align*} 

It is easy to see that exponential OSS in implication form is equivalent to eOSS. Its connection to UGESMO is not that immediate, as the premises in these properties are different. Our results below show, however, that these properties are  equivalent.
\end{remark}

The next theorem establishes a converse Lyapunov theorem for eOSS and proves that eOSS is equivalent to UGESMO.

\begin{theorem}[Lyapunov criterion for eOSS] 
\label{thm:direct eOSS}
   The following are equivalent:
     \begin{enumerate}[(i)]        
\item System~\eqref{eq:linear-system} admits a coercive eOSS Lyapunov function in implication form.

\item System \eqref{eq:linear-system} is eOSS.

\item System \eqref{eq:linear-system} is UGESMO.
\end{enumerate}

\end{theorem}

\begin{proof}
The proof follows from the propositions established below.

(i) $\Rightarrow$ (ii).
Since every coercive eOSS Lyapunov function is also a coercive OSS Lyapunov function, \cite[Theorem 3.2]{chen2025lyapunov} proves OSS of \eqref{eq:linear-system}. Proposition~\ref{prop:equi OSS eOSS LG} shows eOSS.

(ii) $\Rightarrow$ (iii).
    In view of Proposition~\ref{prop:equi OSS eOSS LG}, eOSS is equivalent to eOSS with linear gain. Proposition~\ref{prop:eOSS-ugesmo} proves the implication.

(iii) $\Rightarrow$ (i). This follows from Theorem~\ref{thm:convserse-eOSS}.
\end{proof}

Next, we state and prove the results that we have used to prove the above Lyapunov criterion for eOSS.

\begin{proposition}\label{prop:eOSS-ugesmo}
    If system \eqref{eq:linear-system} is eOSS with linear gain, then it is UGESMO.
\end{proposition}

\begin{proof}
    By assumption, there exist $M,\mu,k>0$ such that for all $x\in X$ and all $t\ge0$
    \begin{align}\label{eq:eoss-assm}
        \|\phi(t,x)\|_X\le M e^{-\mu t}\|x\|_X + k \max\limits_{s\in[0,t]}\|y(s,x)\|_Y. 
    \end{align}
Take $a\in(0, \min\{1,\frac{1}{M+1}\})$ and choose $\rho :=\frac{k}{a}$. Fix $x\in X$ and $\tau\ge0$ and  assume that 
\begin{equation}
\label{eq:UGESMO-cond-in-propeOSS-ugesmo}
\|\phi(t,x)\|_X\geq\rho \|y(t,x)\|_Y \quad\forall t\in[0,\tau].
\end{equation} 
It is our aim to construct $\tilde{M},\tilde{\mu} >0$ for which the estimate \eqref{eq:UGESMO-decay} is satisfied, if \eqref{eq:UGESMO-cond-in-propeOSS-ugesmo} holds.
From \eqref{eq:UGESMO-cond-in-propeOSS-ugesmo} we obtain for all $t\in[0,\tau]$ that
\begin{align*}
k\|y(t,x)\|_Y\le a \|\phi(t,x)\|_X.
\end{align*}
Hence for every interval $[s,s+s_1]\subset[0,\tau]$
\begin{align}\label{eq:maxoninterval}
\max_{t\in[s,s+s_1]} k\|y(t,x)\|_Y\le \max_{t\in[s,s+s_1]}a \|\phi(t,x)\|_X.  
\end{align}

For every $r\in[0,s_1]$, we obtain with \eqref{eq:eoss-assm}
\begin{align}
\label{eq:cocycle-aux}
\|\phi(s+r,x)\|_X=\|\phi(r,\phi(s,x))\|_X &\le M e^{-\mu r }\|\phi(s,x)\|_X + k \max\limits_{t\in[0,r]}\|y(t,\phi(s,x))\|_Y\notag \\
    &= M e^{-\mu r }\|\phi(s,x)\|_X + k \max\limits_{t\in[s,s+r]}\|y(t,x)\|_Y.
\end{align}

Taking the maximum over $r\in[0,s_1]$  and using \eqref{eq:maxoninterval} we obtain 
\begin{align*}
    \max_{t\in[s,s+s_1]} \|\phi(t,x)\|_X\le  M \|\phi(s,x)\|_X+\max_{t\in[s,s+s_1]}a \|\phi(t,x)\|_X.
\end{align*}
Combined with $a<1$, this implies 
\begin{align}\label{eq:stau}
    \max_{t\in[s,s+s_1]} \|\phi(t,x)\|_X\le \frac{M}{1-a} \|\phi(s,x)\|_X.
\end{align} 

Applying \eqref{eq:maxoninterval} again yields
\begin{align*} \max_{t\in[s,s+s_1]} k\|y(t,x)\|_Y\le \frac{aM}{1-a}\|\phi(s,x)\|_X.  \end{align*}

Consequently, by \eqref{eq:cocycle-aux} with $r=s_1$, we have 
\begin{equation}
\label{eq:UGESMO-estimate1}
   \|\phi(s+s_1,x)\|_X\le \Big(M e^{-\mu s_1 } +\frac{aM}{1-a}\Big)\|\phi(s,x)\|_X.  
\end{equation}

As $a < \frac{1}{M+1}$, it holds that $\frac{aM}{1-a}<1$. Now choose $t_1>0$ such that 
\begin{equation*}
  \delta:= M e^{-\mu t_1 } +\frac{aM}{1-a}<1.  
\end{equation*}
We define $\tilde{\mu}:=- \frac{1}{t_1}\log \delta >0$ and it remains to construct a suitable $\tilde{M}>0$ to obtain the desired UGESMO decay estimate \eqref{eq:UGESMO-decay}.

 Consider the two cases $t_1>\tau$ and $t_1 \le \tau$ with $\tau$ chosen before \eqref{eq:UGESMO-cond-in-propeOSS-ugesmo}.

If $\tau<t_1$, then by \eqref{eq:stau} with $s=0$ and $s_1=\tau$, for every $t\in[0,\tau],$
\begin{align*}
    \|\phi(t,x)\|_X\le \frac{M}{1-a}\|x\|_X.
\end{align*}

Since $t\le t_1$, we have $1\le e^{\tilde{\mu}t_1}e^{-\tilde{\mu}t}$ and hence
\begin{align*}
\|\phi(t,x)\|_X\le \frac{M e^{\tilde{\mu}t_1}}{1-a}e^{-\tilde{\mu}t}\|x\|_X =: \tilde{M}_1e^{-\tilde{\mu}t}\|x\|_X, \quad t\in [0,\tau].
\end{align*}
Thus the estimate \eqref{eq:UGESMO-decay} holds with the constants $\tilde{\mu},\tilde{M}_1$. Moreover, these constants only depend on  $M,\mu,k$, and are independent of $\tau$ and $x$.

We now consider the case $\tau\ge t_1$. From \eqref{eq:UGESMO-estimate1} applied to the case $s_1=t_1$ we obtain for all $s\in[0,\tau-t_1]$ that 
\begin{equation}
\label{eq:UGESMO-estimate2}
   \|\phi(s+t_1,x)\|_X\le \delta\|\phi(s,x)\|_X. 
\end{equation}
Let $t\in[0,\tau]$ and $t=nt_1+r_1$ with $n\in \N \cup\{0\}$ and $r_1\in[0,t_1).$ From \eqref{eq:UGESMO-estimate2} we obtain $\|\phi((\ell+1)t_1,x)\|_X\le \delta\|\phi(\ell t_1,x)\|_X$, $\ell=0,\ldots,n-1$ and by iteration we obtain
\begin{equation*}
  \|\phi(nt_1,x)\|_X\le \delta^n\|x\|_X.  
\end{equation*}

Using \eqref{eq:stau} with $s=nt_1$ and $s_1=r_1$, we get 
\begin{align*} \|\phi(t,x)\|_X\le \frac{M}{1-a}\|\phi(nt_1,x)\|_X\le  \frac{M}{1-a}\delta^n\|x\|_X.   \end{align*}
Since $nt_1 > t-t_1,$ we have $\delta^n\le \frac{1}{\delta}e^{\frac{\ln \delta}{t_1}t} = \frac{1}{\delta} e^{-\tilde{\mu}t}$. Setting $\tilde{M}_2 := \frac{M}{(1-a)\delta}$, we obtain $\|\phi(t,x)\|_X\le \tilde{M}_2 e^{-\tilde{\mu}t}\|x\|_X$ for all $t\in[0,\tau].$ 
The desired UGESMO decay estimate \eqref{eq:UGESMO-decay} now holds with the constructed $\tilde{\mu}$ and $\tilde{M}:= \max\{ \tilde{M}_1,\tilde{M}_2\}$.
\end{proof}

Now we can prove the first converse Lyapunov theorem for eOSS for linear infinite-dimensional systems with bounded output operators.

\begin{theorem}[Converse eOSS Lyapunov theorem]\label{thm:convserse-eOSS}
Suppose that system \eqref{eq:linear-system} is UGESMO.
Then it admits a coercive eOSS Lyapunov function in implication form. Moreover,  this function can be chosen to be $1$-absolutely homogeneous, with linear Lyapunov gain and $p=1$. 
In particular, this applies if system \eqref{eq:linear-system} is eOSS.
\end{theorem}

\begin{proof}
As eOSS implies UGESMO by Proposition~\ref{prop:eOSS-ugesmo}, 
it is sufficient to prove that UGESMO implies the assertions. So assume that
system \eqref{eq:linear-system} is UGESMO with some $\rho, M,\mu>0$. We denote the semigroup associated to \eqref{eq:linear-system} by $T(\cdot)$.
Define $\mathcal{E}_\rho$ as the set in which the premise of Definition~\ref{UGESMO} holds with strict inequality. We thus consider the sets
\begin{align}
\label{eq:E-and-D-sets}
\mathcal{E}_\rho:=\{x\in X\mid \|x\|_X>\rho\|Cx\|_{Y}\}, \quad\mathcal{D}_{\rho}:=X\setminus\mathcal{E}_{\rho}.
\end{align}
Since $C$ is linear,   $\mathcal{D}_{\rho}$ is closed under scalar multiplication, in the sense that if $x \in \mathcal{D}_\rho$, then $a x \in \mathcal{D}_\rho$ for all $a \in\R$. Similarly, $\mathcal{E}_\rho\cup\{0\}$ is 
closed under scalar multiplication.

Note that $\mathcal{D}_{\rho}$ is closed and nonempty since $0\in\mathcal{D}_{\rho}$, and thus $\mathcal{E}_\rho$ is open.
For every $x\in X$, define the first entrance time into the set $\mathcal{D}_{\rho}$ by
\begin{align}
\label{eq:lambda-map}
\lambda(x):=\inf\{t\geq0\mid T(t)x\in\mathcal{D}_{\rho}\},
\end{align}
and set $\lambda(x):=\infty$, if the trajectory never enters $\mathcal{D}_{\rho}$. We have 
\begin{align}
\label{eq:lambda-property}
\lambda(T(t)x)=\lambda(x)-t, \quad x\in\mathcal{E}_\rho, \ t\in[0,\lambda(x)).
\end{align}
We regard $\lambda$ as taking values in $[0,\infty]$, with $\infty-t:=\infty$, $t\geq 0$. Hence, if $\lambda(x)<\infty$, then the map
$t\mapsto \lambda(T(t)x)$
is continuous on $[0,\lambda(x))$; if $\lambda(x)=\infty$, then $\lambda(T(t)x)\equiv\infty$.

Note furthermore that $\lambda$ is invariant under scalar multiplication, i.e., for all $a\in \R\setminus\{0\}$ and all $x
\in X$, since $\mathcal{D}_{\rho}$ is closed under scalar multiplication,
\begin{align}\label{eq:lambda-homo}
\lambda(ax)=\inf\{t\geq0\mid a (T(t)x)\in\mathcal{D}_{\rho}\}=\lambda(x).
\end{align}

\vspace{1mm}\noindent\emph{Step 1: Construction of the eOSS Lyapunov candidate.}
For $x\in\mathcal{E}_\rho,$ define
\begin{align*}
    W(x):=\int_0^{\lambda(x)} \|T(s)x\|_X \, \mathrm{d} s.
\end{align*}
By the definition of UGESMO, we get for $x\in\mathcal{E}_\rho$

\begin{align}\label{eq:W-coer}
    0\le W(x) = \int_0^{\lambda(x)} \|T(s)x\|_X \, \mathrm{d} s \stackrel{\eqref{eq:UGESMO-decay}}{\le} \int_0^{\lambda(x)} Me^{-\mu s}\|x\|_X \, \mathrm{d} s \leq \frac M\mu\|x\|_X.
\end{align}

For $x\in\mathcal{E}_\rho$ and  $0<\eta<\lambda(x)$, we obtain with \eqref{eq:lambda-property} that
\begin{equation}
\label{eq:Wcontinuous}
    \begin{aligned}
W(T(\eta)x)=\int_0^{\lambda(x)-\eta} \|T(s+\eta)x\|_X \, \mathrm{d} s&=\int_\eta^{\lambda(x)} \|T(s)x\|_X \, \mathrm{d} s\\
&=W(x)-\int_0^{\eta} \|T(s)x\|_X \, \mathrm{d} s.
\end{aligned}
\end{equation}

As $\mathcal{E}_\rho$ is open and $t\mapsto T(t) x$ is continuous, this shows that $t\mapsto W(T(t)x)$ is continuous on $[0,\lambda(x))$ provided that $x\in \mathcal{E}_\rho$.

 The Lie derivative of $W$ in $x\in \mathcal{E}_\rho$ is now easily computed as
\begin{align}\label{eq:w-dini}
\dot{W}(x)=\mathop{\overline{\lim}}\limits_{\eta\rightarrow+0}
\frac{1}{\eta} (W(T(\eta)x)-W(x))
=\mathop{\overline{\lim}}\limits_{\eta\rightarrow+0} -\frac{1}{\eta}\int_0^{\eta} \|T(s)x\|_X \, \mathrm{d} s=-\|x\|_X.
\end{align}

Fix $\rho_1, \rho_2$ with $\rho<\rho_1<\rho_2.$  Since the function $r\mapsto\frac{r}{1+r}$ is strictly increasing on $(0, \infty),$ we have
\begin{align*}
    \frac{\rho}{1+\rho}<\frac{\rho_1}{1+\rho_1}<\frac{\rho_2}{1+\rho_2}.
\end{align*}

Choose a continuous function $\theta: [0,1] \to [0,1]$ such that
\begin{align*}
    \theta(s)=
    \begin{cases}
      0, &s\le \frac{\rho}{1+\rho},\\
        1, &s\ge \frac{\rho_1}{1+\rho_1}.
    \end{cases}
\end{align*}
 Then, for every nonzero $x\in X$, the following implication holds
\begin{align}\label{eq:rho2}
\|x\|_X\ge\rho_2\|Cx\|_Y\ge\rho_1\|Cx\|_Y\quad \Rightarrow\quad \theta\left(\frac{\|x\|_X}{\|x\|_X+\|Cx\|_Y}\right)=1.
\end{align}

Since $T(\cdot)$ is a $C_0$-semigroup, there exist $M_1, \omega_0>0$ such that
\begin{align*} \|T(t)\|\le M_1 e^{\omega_0 t}, \quad t\ge0. \end{align*}
Choose $\omega_1>\omega_0$ and define
\begin{equation} 
\label{eq:normdef}
N(x):=\max_{s\ge0}e^{-\omega_1 s}\|T(s)x\|_X,\quad x\in X. \end{equation}

It is well-known, \cite[Chapter II]{engel2000one}, that $N$ is an equivalent norm and in particular continuous. Also
\begin{align}\label{eq:N-coer}
\|x\|_X\le N(x)\le M_1\|x\|_X, \quad x\in X.
\end{align}
 In particular, $t\mapsto N(T(t)x)$ is continuous. Moreover, $N(T(\eta)x)\le e^{\eta \omega_1 }N(x)$ for $\eta>0, x\in X$, and hence for all $x\in X$ the Lie derivative may be estimated as
\begin{align}\label{eq:n-dini}
\dot{N}(x)= \mathop{\overline{\lim}}\limits_{\eta\rightarrow+0} \frac{N(T(\eta)x)-N(x)}{\eta}  &\le \mathop{\overline{\lim}}\limits_{\eta\rightarrow+0} \frac{e^{\eta \omega_1 }N(x)-N(x)}{\eta} \notag\\   
&=\omega_1N(x) 
\le \omega_1 M_1\|x\|_X.
\end{align}

 Choose $K>0$ with
\begin{equation}\label{eq:K}
    K>\omega_1 M_1
\end{equation}
and define
\begin{equation*}
    V(x):=
    \begin{cases}
    N(x)+K\theta\Bigl(\frac{\|x\|_X}{\|x\|_X+\|Cx\|_Y}\Bigr)W(x),&\quad x\in \mathcal{E}_\rho,\\
        N(x),&\quad x\in \mathcal{D}_\rho.
    \end{cases}
\end{equation*}

\vspace{1mm}\noindent\emph{Step 2: Verification of the eOSS Lyapunov  function properties for $V$.}
Let us verify that $V$ satisfies (i), (ii) and (iii') in Definition~\ref{def:eOSS LF-dissipative}.

\emph{(i): Sandwich bounds}. For all $x\in X$, \eqref{eq:W-coer} and \eqref{eq:N-coer} and $0\le \theta \le1$ yield
\begin{align*}
    \|x\|_X\le V(x)\le \bigl(M_1+\frac{KM}{\mu}\bigr)\|x\|_X.
\end{align*}

\emph{(ii): Continuity of $V$ along trajectories}. Fix $x\in X$ and $t\ge0$. We prove that $s\mapsto V(T(s)x)$ is continuous at $s=t.$

If $T(t)x\in \operatorname{int}\mathcal{D}_\rho$, this follows  from the continuity of the norm $N$.

If $T(t)x\in \mathcal E_\rho$, by the openness of $\mathcal E_\rho$ there exists a neighborhood of $t$ on which $T(s)x\in \mathcal E_\rho$. For $s_1<s_2$ in this neighborhood, the entrance-time shift yields
\[
W(T(s_1)x)-W(T(s_2)x)
=
\int_{s_1}^{s_2}\|T(s)x\|_X\,ds.
\]
Hence, $s\mapsto W(T(s)x)$ is continuous at $s=t$, and the continuity of the remaining terms is obvious.

It remains to treat the case $T(t)x\in  \partial  \mathcal{D}_\rho = \partial \mathcal{E}_\rho.$   Since  $T(t)x\in \partial  \mathcal{D}_\rho $  and $\mathcal{D}_\rho$  is closed, we have $T(t)x\in \mathcal{D}_\rho$, and hence $V(T(t)x)=N(T(t)x).$  For $s$ close to $t$,  $T(s)x$  may belong to either $\mathcal{D}_\rho $ or $\mathcal{E}_\rho $. Therefore, continuity at $s=t$ must be verified across the boundary.

Let $(t_n)_{n\in \mathbb{N}}\subset\mathbb{R}_+$ be an arbitrary sequence with $t_n\to t$. Define
\begin{align*} I_D:=\{n\in\mathbb{N}\mid T(t_n)x\in\mathcal{D}_\rho\}, \quad I_E:=\mathbb{N}\setminus I_D.  \end{align*}

We show that $V(T(t_n)x)\to V(T(t)x).$

 For every subsequence $(t_{n_j})_{j\in\mathbb N}$ of $(t_n)_{n\in\mathbb N}$ whose indices belong to $I_D$, we get $V(T(t_{n_j})x) = N(T(t_{n_j})x)$ and the convergence follows directly from the continuity of the norm $N$.

Assume $(t_{n_j})_{j\in\mathbb N}$ is a subsequence of $(t_n)_{n\in\mathbb N}$ whose indices belong to $I_E$. For every $j\in \N$ we have
\begin{align*}
V(T(t_{n_j})&x)
=
N(T(t_{n_j})x)  +
K\theta\left(
\frac{\|T(t_{n_j})x\|_X}
{\|T(t_{n_j})x\|_X+\|CT(t_{n_j})x\|_Y}
\right)
W(T(t_{n_j})x).
\end{align*}
By continuity of $N$, the first term converges to $N(T(t)x)$. We are going to show that the second term converges to $0$.

 If $T(t)x\neq0$, then, as
$j\to \infty$
\begin{align*} 
\frac{\|T(t_{n_j})x\|_X}
{\|T(t_{n_j})x\|_X+\|CT(t_{n_j})x\|_Y}\to\frac{\|T(t)x\|_X}
{\|T(t)x\|_X+\|CT(t)x\|_Y}=\frac{\rho}{1+\rho}.  \end{align*}
Moreover, by \eqref{eq:W-coer}
\begin{align}
    W(T(t_{n_j})x)\le \frac{M}{\mu}\|T(t_{n_j})x\|_X,
\end{align}
and the right-hand side is bounded because $T(t_{n_j})x\to T(t)x$. 
 
Since $\theta$ is continuous and $\theta\bigl(\frac{\rho}{1+\rho}\bigr)=0$, we obtain 
\begin{align*} K\theta\left(
\frac{\|T(t_{n_j})x\|_X}
{\|T(t_{n_j})x\|_X+\|CT(t_{n_j})x\|_Y}
\right)
W(T(t_{n_j})x)\to 0, \quad { j\to \infty}. \end{align*}

If $T(t)x=0$, then $T(t_{n_j})x\to0$. Since $0\le \theta\le 1$, by \eqref{eq:W-coer} we have, as $j\to \infty$,
\begin{align*}
0&\le K\theta\left(
\frac{\|T(t_{n_j})x\|_X}
{\|T(t_{n_j})x\|_X+\|CT(t_{n_j})x\|_Y}
\right)
W(T(t_{n_j})x)\le \frac{KM}{\mu}\|T(t_{n_j})x\|_X\to 0.
\end{align*}
Thus, 
\begin{align*} V(T(t_{n_j})x) \to V(T(t)x),\quad { j\to \infty}.  \end{align*}

Thus $\lim_{j\to\infty}V(T(t_{n_j})x) = V(T(t)x)$.

As the limits for subsequences whose indices belong to $I_D$ respectively $I_E$ coincide, this shows that also for arbitrary subsequences  $(t_{n_j})_{j\in\N}$ we have $\lim_{j\to\infty}V(T(t_{n_j})x) = V(T(t)x)$ and this shows the continuity of $V$ along the trajectory $(T(s)x)_{s\geq 0}$ in $s=t$.

\emph{(iii'): The decay estimate}.  We verify the implication condition \eqref{eq:aboss Implication}. To this end, let $x\in X$ satisfy
$ \|x\|_X\ge \rho_2\|Cx\|_Y$.

If $x=0$, then the desired implication inequality
is immediate since $T(t)0=0$, $t\geq 0$, and $V(0)=0$. Hence, assume that $x\ne0$.
Then
\begin{align}\label{eq:rho_1}
\|x\|_X>\rho_1\|Cx\|_Y.
\end{align}
Indeed, if $Cx\ne0$, this follows from
$\rho_2>\rho_1$; if $Cx=0$, it follows from $\|x\|_X>0$. In particular,
$x\in\mathcal E_\rho$. Moreover, by \eqref{eq:rho2} and \eqref{eq:rho_1}
\begin{align*} 
\theta\left(\frac{\|x\|_X}{\|x\|_X+\|Cx\|_Y}\right)=1. \end{align*}
 By the continuity of $t\mapsto T(t)x$, there exists a sufficiently small $\eta^*>0$ such that 
 \[
 \|T(\eta)x\|_X>\rho_1\|CT(\eta)x\|_Y, \quad \eta \in [0,\eta^*).
 \]
 It follows that
\begin{align*} \theta\left(\frac{\|T(\eta)x\|_X}{\|T(\eta)x\|_X+\|CT(\eta)x\|_Y}\right) =1, \quad { \eta \in [0,\eta^*)}. \end{align*}

As $x\in \cal E_\rho$, we get using \eqref{eq:w-dini} and \eqref{eq:n-dini}
\begin{align*}
  \dot{V}(x)=\mathop{\overline{\lim}}\limits_{\eta\rightarrow+0} \frac{N(T(\eta)x)+KW(T(\eta)x)-N(x)-KW(x)}{\eta} 
  &\le \dot{N}(x)+K \dot{W}(x)\\
 &\le -(K-\omega_1M_1)\|x\|_X.
\end{align*}

Therefore, by \eqref{eq:K}, the required implication condition \eqref{eq:aboss Implication} holds.
We have shown that $V$ is a coercive eOSS Lyapunov function in implication form with linear Lyapunov gain $\chi(r) = \rho_2 r$, $r\geq 0$, $\tilde{\alpha} = K-\omega_1M_1$, and $p=1$.

\vspace{1mm}\noindent\emph{Step 3: Absolute homogeneity.}
Finally, we verify absolute homogeneity. 
By \eqref{eq:lambda-homo}, for all $a\in \R\setminus \{0\}$ and all $x
\in \mathcal{E}_\rho$,
\begin{align*} W(ax)=\int_0^{\lambda(ax)} \|T(s)(ax)\|_X \, \mathrm{d} s  \stackrel{\eqref{eq:lambda-homo}}{=} \int_0^{\lambda(x)} \|T(s)(ax)\|_X \, \mathrm{d} s =|a|W(x). \end{align*}
For every $a\in \R\setminus \{0\}$ and $x\in \mathcal{E}_\rho$, we have 
\begin{align*} \frac{\|ax\|_X}{\|ax\|_X+\|C(ax)\|_Y}=\frac{\|x\|_X}{\|x\|_X+\|Cx\|_Y},  \end{align*}
and hence
\begin{align*} \theta\left(\frac{\|ax\|_X}{\|ax\|_X+\|C(ax)\|_Y}\right)=\theta\left(\frac{\|x\|_X}{\|x\|_X+\|Cx\|_Y}\right). \end{align*}

It is obvious that $N(ax)=|a|N(x)$ for all $a\in\R$ and all $x\in X$.

Combining the above identities, and noting that the case $a=0$ follows from $V(0)=0$, we obtain
\begin{align*}
    V(ax)=|a|V(x),\quad a\in \R,\ x\in X.
\end{align*}
Thus $V$ is $1$-absolutely homogeneous, and the proof is complete.
\end{proof}

\begin{remark}
 Theorem~\ref{thm:convserse-eOSS} differs from the OSS Lyapunov characterization of \cite{sontag1997output} for nonlinear ODE systems both in the formulation and in the proof technique. In \cite{sontag1997output}, Sontag and Wang show that every sufficiently regular OSS ODE system admits a coercive OSS Lyapunov function in dissipation form (which is also a coercive OSS Lyapunov function in implication form). 
 As we discuss in Section~\ref{sec:counterexam} below, in this formulation the result does not hold for linear infinite-dimensional systems, and we can merely ensure the existence of an OSS Lyapunov function in implication form. Even after weakening the conclusion of \cite{sontag1997output}  to implication form, the finite-dimensional proof does not directly carry over to the infinite-dimensional case. An essential step in our proof is to establish continuity along solutions at the boundary between the two regions for the constructed Lyapunov function. The technique by which this is achieved in \cite{sontag1997output} relies on local compactness and the existence of smooth mollifiers. It therefore cannot be transferred to the infinite-dimensional setting. Whether our proof technique can be extended to the nonlinear infinite-dimensional case is not fully clear right now. Some preliminary lemmas were reported in \cite{chen2025lyapunov}, but the main work remains to be done. 
\end{remark}

\section{From detectability to eOSS and back}
\label{sec: exp detectability}

Theorem~\ref{thm:direct eOSS} shows the equivalence of eOSS and of the existence of an eOSS Lyapunov function in an implication form. However, it does not give us an easy-to-use condition for constructing such Lyapunov functions to prove eOSS. Hence, it is also important to identify structural system properties that guarantee the existence of such functions.
One such property is exponential detectability, as defined, e.g., in \cite[Definition 8.1.1]{Curtain2020}.

The aim of this section is twofold. First, we show that exponential detectability implies the existence of a coercive eOSS Lyapunov function in  dissipation form, and hence eOSS itself. 
Afterwards, we summarize the results of this paper in Theorem~\ref{thm:single direction}, giving a precise and exhaustive description of relations between eOSS, existence of an eOSS Lyapunov function, and various detectability concepts to each other.

\begin{definition}\label{def:exp detect}    
System~\eqref{eq:linear-system} is called \emph{exponentially detectable} if there exists
$L\in\mathcal L(Y,X)$ such that $A+LC$ generates a uniformly exponentially stable semigroup $T_{LC}(\cdot)$. 
\end{definition}

The construction of the following coercive eOSS Lyapunov function is motivated by the approach in \cite[Proposition 7]{mironchenko2018lyapunov}, where a similar construction was used for linear systems with bounded input operators under a $0$-UGAS assumption. Here, we adapt this idea to linear systems with bounded output operators under exponential detectability.

\begin{proposition}\label{prop:coer eOSS LF}
  Assume that system \eqref
{eq:linear-system} is exponentially detectable and let 
$T_{LC}(\cdot)$ be as in Definition~\ref{def:exp detect} satisfying \eqref{eOSS} for some $\mu > 0$, $\gamma=0$. For any $\eta \in (0, \mu)$, define
\begin{equation}\label{eq:coer LY}
V^\eta(x) := \max_{s \ge 0} \left\| e^{\eta s} T_{LC}(s) x \right\|_X.
\end{equation}
Then 
    $V^\eta$ is a coercive eOSS Lyapunov function in dissipation form with  linear  $\sigma$ and $p=1$  for system \eqref
{eq:linear-system}.
\end{proposition}
\begin{proof}
As in the construction in \eqref{eq:normdef}, $V^\eta$ is an equivalent norm and in particular coercive, satisfying 
\begin{align}\label{eq:ii0}
\|x\|_X\le V^\eta(x)\le M \|x\|_X,\quad x\in X.
\end{align}
 for a suitable constant $M$.

 It follows from the definition of $V^\eta$ and the semigroup property that
\begin{align}\label{eq:ii1}
V^\eta(T_{LC}(t) x) \le e^{-\eta t} V^\eta(x), \quad  t \ge 0.
\end{align}

For the trajectory $x(\cdot)=\phi(\cdot,x)\in C([0,t],X)$ with initial state $x$ and $t>0$, the function
$y(\cdot):=Cx(\cdot)$ belongs to $C([0,t],Y)$. Then by continuity, it holds that
\begin{equation}\label{eq:ii2}
\lim_{\tau \to 0^+} \frac{1}{\tau} \int_0^\tau T_{LC}(\tau -s) Ly(s)\, \mathrm ds = Ly(0) = LC x. 
\end{equation}

As $V^\eta$ is a norm, we have
\begin{equation}\label{eq:ii3}
    V^\eta(x-y)\le V^\eta(x)+V^\eta(y),\quad x,y\in X.
\end{equation}

Rewrite  the system \eqref{eq:linear-system} as
\begin{align}\label{eq:re linear}
\dot{x}=(A+LC)x-LCx.
\end{align}
For $t\ge 0$, the mild solution is 
\begin{equation}\label{eq:variation}
\phi(t,x)=
T_{LC}(t)x
-\int_0^t T_{LC}(t-s)Ly(s)\, \mathrm ds.
\end{equation}

Using step-by-step \eqref{eq:variation}, \eqref{eq:ii3}, \eqref{eq:ii1}, \eqref{eq:ii2}, and \eqref{eq:ii0}, we obtain 
\begin{align*}
\dot{V}^\eta(x)
&= \mathop{\overline{\lim}}\limits_{\tau\rightarrow+0}\frac{1}{\tau} \left( V^\eta(\phi(\tau,x)) - V^\eta(x) \right)\\
&\stackrel{\eqref{eq:variation}}{=} \mathop{\overline{\lim}}\limits_{\tau\rightarrow+0}\frac{1}{\tau} \biggl( V^\eta\Bigl( T_{LC}(\tau) x - \int_0^\tau T_{LC} (\tau-s) Ly(s) \, \mathrm ds \Bigr) - V^\eta(x) \biggr) \\
& \stackrel{\eqref{eq:ii3}}{\le}\mathop{\overline{\lim}}\limits_{\tau\rightarrow+0}\frac{1}{\tau} \biggl( V^\eta(T_{LC}(\tau) x) + V^\eta\Bigl( \int_0^\tau T_{LC} (\tau-s)  L y(s)\, \mathrm ds \Bigl)- V^\eta(x) \biggr) \\
&\stackrel{\eqref{eq:ii1}}{\le}\mathop{\overline{\lim}}\limits_{\tau\rightarrow+0}\frac{1}{\tau}\biggl( \left( e^{-\eta \tau} - 1 \right) V^\eta(x)+ V^\eta\Bigl( \int_0^\tau T_{LC} (\tau-s)  L y(s)\, \mathrm ds \Bigr) \biggr) \\
&\stackrel{\eqref{eq:ii2}}{\le} -\eta V^\eta(x) + V^\eta
(L y(0))\\
&\stackrel{\eqref{eq:ii0}}{\le }-\eta \|x\|_X +M \|L\|_{\mathcal{L}(Y,X)}\|Cx\|_Y.
\end{align*}
Thus, if $L\neq0$, $V^\eta$ is a coercive eOSS Lyapunov function in dissipation form with linear gain given by $\sigma(r)=M\|L\|_{\mathcal{L}(Y,X)}r$, $r\geq 0$, and $p=1$. If $L=0$, then any linear $\sigma\in\mathcal{K}_\infty$ may be chosen, and $V^\eta$ is a coercive UES Lyapunov function.
\end{proof}

Following the classical detectability condition, see e.g. \cite[p. 317]{sontag2013mathematical}, later referred to as zero-detectability in \cite{krichman2001input}, we strengthen zero-detectability by requiring exponential convergence, and define exponential zero-detectability as follows.
\begin{definition} 
\label{def:exp-zero-detect}
    System \eqref{eq:linear-system} is called 
\emph{exponentially zero-detectable}, if there exist
$M,\omega>0$ such that for all $x\in X$ 
\[
y(\cdot,x)\equiv 0 \ \Rightarrow\ \|\phi(t,x)\|_X\le Me^{-\omega t}\|x\|_X,
\quad t\ge0.
\]
\end{definition}

By Proposition~\ref{prop:equi OSS eOSS LG}, OSS is equivalent to exponential OSS, which implies
\begin{corollary}
\label{cor:OSS implies exp zero}
    If system \eqref{eq:linear-system} is OSS, then it is exponentially zero-detectable.
\end{corollary}

We are now ready to summarize the results of this paper.
\begin{theorem}
\label{thm:single direction}
   Consider the following statements:
     \begin{enumerate}[(i)]        
\item System~\eqref{eq:linear-system} is exponentially detectable. 

\item System~\eqref{eq:linear-system} admits a coercive eOSS Lyapunov function in dissipation form with linear Lyapunov gain.

\item System~\eqref{eq:linear-system} admits a coercive eOSS Lyapunov function in implication form with linear Lyapunov gain.

\item System \eqref{eq:linear-system} is eOSS with linear gain.

\item System \eqref{eq:linear-system} is eOSS.

\item System \eqref{eq:linear-system} is OSS with linear gain.

\item  System \eqref{eq:linear-system} is OSS.

\item {System \eqref{eq:linear-system} is UGESMO.}

\item System \eqref{eq:linear-system} is exponentially zero-detectable.
\end{enumerate}
Then 
\[
\text{(i) $\Rightarrow$ (ii) $\Rightarrow$ (iii) $\Leftrightarrow$ (iv)
$\Leftrightarrow$ (v) $\Leftrightarrow$ (vi) $\Leftrightarrow$ (vii)
$\Leftrightarrow$ (viii) $\Rightarrow$ (ix).}
\]

\end{theorem}
\begin{proof}
   The implication (i) $\Rightarrow$ (ii) follows from Proposition~\ref{prop:coer eOSS LF}. For (ii) $\Rightarrow$ (iii), Corollary~\ref{cor:dissi-imp} yields a coercive eOSS Lyapunov function in implication form. Then Theorem~\ref{thm:direct eOSS}, Proposition \ref{prop:equi OSS eOSS LG}, and Theorem~\ref{thm:convserse-eOSS} give (iii). The equivalence (iii) $\Leftrightarrow$ (iv) $\Leftrightarrow$ (v) $\Leftrightarrow$ (vi) $\Leftrightarrow$ (vii) $\Leftrightarrow$ (viii) is shown in Proposition \ref{prop:equi OSS eOSS LG} and Theorem~\ref{thm:direct eOSS}. The implication (viii) $\Rightarrow$ (ix) follows from Corollary \ref{cor:OSS implies exp zero}.
\end{proof}

The implication structure is summarized in Fig.~\ref{fig:Result}.

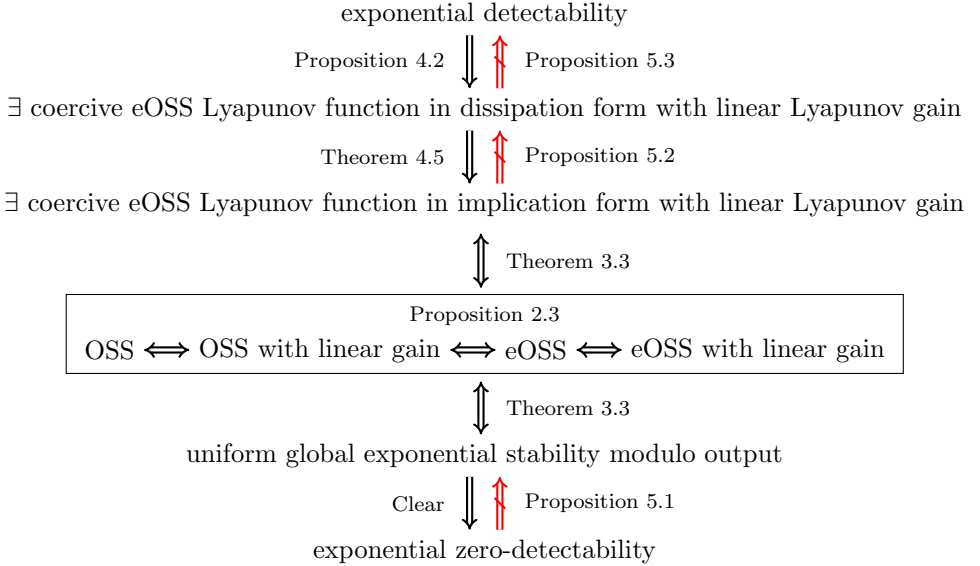
\begin{figure}[htbp]    
\vspace{.3cm}    
\centering
\begin{tikzpicture}
\node (a) at (-6,0) {OSS};
\node[right=.6cm of a.east](b) {OSS with linear gain};
\node[right=.6cm of b.east](c) {eOSS};
\node[right=.6cm of c.east](d) {eOSS with linear gain};

\node[fit=(a) (d),draw=none,inner ysep=0mm](dummyBlockEquiv) {}; 
 
\node[above=0.1cm of dummyBlockEquiv,anchor=base] (titleEquiv) {\footnotesize Proposition \ref{prop:equi OSS eOSS LG}};        

\node[rectangle,fit=(a) (d) (titleEquiv),draw=black,inner ysep=.3mm](BlockEquiv) {};

\node[above = 0.9cm of BlockEquiv] (e2) {$\exists$ coercive eOSS Lyapunov function in implication form with linear Lyapunov gain};

\node[above = 0.7cm of e2](e1)
{$\exists$ coercive eOSS Lyapunov function in dissipation form with linear Lyapunov gain};

\node[above = 0.7cm of e1] (f) {exponential detectability};

 \node[below = 0.8cm of BlockEquiv] (h) {uniform global exponential stability modulo output};

 \node[below = 0.7cm of h] (g) {exponential zero-detectability};

 \path 
(a) edge[thick,double,double equal sign distance,{Implies[]}-{Implies[]}] (b) 

(b) edge[thick,double,double equal sign distance,{Implies[]}-{Implies[]}] (c)        
(c) edge[thick,double,double equal sign distance,{Implies[]}-{Implies[]}] (d);        


\path ([shift={(-0.05cm,-0.1cm)}]e2.south) edge[thick,double,double equal sign distance,{Implies[]}-{Implies[]}]node[right=.2cm]{\footnotesize Theorem~\ref{thm:direct eOSS}}([shift={(-0.05cm,0.08cm)}]BlockEquiv.north);

  \path([shift={(-0.05cm,-0.1cm)}]BlockEquiv.south) edge[thick,double,double equal sign distance,{Implies[]}-{Implies[]}]  node[right=.2cm]{\footnotesize  Theorem~\ref{thm:direct eOSS}} ([shift={(-0.05cm,0cm)}]h.north)

([xshift=-0.2cm]e1.south) edge[thick,double,double equal sign distance,-{Implies[]}] node[left=.2cm]{\footnotesize Theorem~\ref{thm:single direction}} ([xshift=-0.2cm] e2.north)

([xshift=0.2cm]e2.north) edge[thick,double,double equal sign distance,-{Implies[]},degil,draw=red] node[right=.2cm]{\footnotesize Proposition \ref{prop:no eOSS LF}} ([xshift=0.2cm]e1.south)

([xshift=-0.2cm]h.south) edge[thick,double,double equal sign distance,-{Implies[]}] node[left=.2cm]{\footnotesize  Clear} ([xshift=-0.2cm] g.north)

([xshift=0.2cm]g.north) edge[thick,double,double equal sign distance,-{Implies[]},degil,draw=red] node[right=.2cm]{\footnotesize Proposition \ref{prop: not eOSS}} ([xshift=0.2cm]h.south)

([xshift=-0.2cm]f.south) edge[thick,double,double equal sign distance,-{Implies[]}] node[left=.2cm]{\footnotesize Proposition \ref{prop:coer eOSS LF} } ([xshift=-0.2cm] e1.north)

([xshift=0.2cm]e1.north) edge[thick,double,double equal sign distance,-{Implies[]},degil,draw=red] node[right=.2cm]{\footnotesize Proposition \ref{prop:not-dete}} ([xshift=0.2cm]f.south);
\end{tikzpicture}
\caption{Summary of the statements of  Theorem~\ref{thm:single direction}. The crossed arrows indicate converse
implications which fail in general; the corresponding counterexamples are
given in Section~\ref{sec:counterexam}.} 
\label{fig:Result}
\end{figure}

\begin{remark}
In \cite{chen2025lyapunov}, it was shown that uniform global asymptotic stability modulo output implies the vanishing output vanishing state property (see \cite{chen2025lyapunov} for the definition). It is possible to introduce exponential versions of the concept and put it into the general picture of Figure~\ref{fig:Result}. But we do not follow this path here.
\end{remark}

\section{Counterexamples}\label{sec:counterexam}

Theorem~\ref{thm:single direction} gives a chain of implications from
exponential detectability to exponential zero-detectability. The purpose of this section is to show that, in general, the converse implications
cannot be expected. We present three counterexamples:
\begin{itemize}
   \item[(i)] exponential zero-detectability does not imply eOSS;
    \item[(ii)] eOSS does not imply the existence of an eOSS Lyapunov function in dissipation form;
    \item[(iii)] the existence of a coercive eOSS Lyapunov function in dissipation form with linear Lyapunov gain does not imply exponential detectability.
\end{itemize}

\subsection{Exponential zero-detectability does not imply  eOSS}

The first example separates exponential zero-detectability from eOSS.

Let $X=L^2(0,1)$ be the space of square integrable Lebesgue measurable complex-valued functions on $(0,1)$ and $Y=\mathbb C$. Consider the system
\begin{equation}\label{eq:example system}
\dot x(t)=Ax(t), \quad y(t)=Cx(t),
\end{equation}
where $A\in\mathcal L(X) $ and  $C\in\mathcal L(X,Y)$
 are defined by
\[
(Af)(s)=is f(s), \ \ s\in (0,1); \quad
Cf=\int_0^1 f(s)\, \mathrm ds.\]
It is well-known \cite[Example II.5.8(i)]{engel2000one} that the operator $A$ generates the  $C_0$-group  $T(\cdot)$ given by 
\begin{center}
$(T(t)f)(s)=e^{ist}f(s)$,\quad $s\in (0,1)$,\ $t\in\R$.     
\end{center}
It is easy to see that $T$ is isometric, i.e., $\|T(t)f\|_X = \|f\|_X$ for all $f\in X$ and all $t\geq 0$.

\begin{proposition}\label{prop: not eOSS}
 The following hold: 
\begin{enumerate}
    \item [(i)]  System \eqref{eq:example system} is exponentially zero-detectable.
    \item [(ii)]  System \eqref{eq:example system} is not eOSS.
\end{enumerate}
\end{proposition}

\begin{proof} 
(i). We first show that the system \eqref{eq:example system} is exponentially zero-detectable.
Assume that $x\in X$ satisfies
\[
CT(t)x = \int_0^1 e^{ist}x(s)\, \mathrm ds = 0
\quad \forall t\ge 0.
\]

Define
\[
F(z):=\int_0^1 e^{isz}x(s)\, \mathrm ds,\quad z \in\C.
\]

Let $x\in L^2(0,1)$ and extend it by zero to 
$\tilde x\in L^1(\mathbb R)\cap L^2(\mathbb R)$.
Then
\begin{align*} 
F(z)=
\int_0^1 e^{isz}x(s)\, \mathrm ds
    =
\int_{\mathbb R}\tilde x(s)e^{isz}\, \mathrm ds .
 \end{align*}
By 
\cite[Section~19.1(b)]{Rudin1987}, $F$ is an entire holomorphic function.

Since 
$F(t)=0$ for all $t\ge0$,
the identity theorem 
\cite[Theorem~10.18]{Rudin1987} implies $F\equiv0$. Hence, the Fourier transform of 
$\tilde x$
 vanishes identically, and by the uniqueness theorem
\cite[Theorem~9.12]{Rudin1987} we obtain $\tilde x=0$ a.e.,
so  $x=0$ in $L^2(0,1)$.
 Therefore, the implication
\[
CT(t)x\equiv 0 \ \Rightarrow\ \|T(t)x\|_X\le Me^{-\omega t}\|x\|_X,
\quad t\ge0,
\]
holds trivially for any $M,\omega>0$, and the system is exponentially zero-detectable.

(ii). For 
$n\in\mathbb N$ define $x_n(s):=e^{isn}$, $s\in (0,1)$. Then $\|x_n\|_X=1$
 and
\begin{align*}
CT(t)x
_n=\int_0^1 e^{ist}e^{isn}\, \mathrm ds
=\frac{e^
{i(n+t)}-1}{i(n+t)},\quad t\ge0.
\end{align*}
Thus, for any $t\ge0$ and for all $\tau\in[0,t]$, we get
\begin{align*}
    |CT(\tau)x_n|=\Big|\frac{e^
{i(n+\tau)}-1}{i(n+\tau)}\Big|\le \frac{2}{|n+\tau|}\le \frac{2}{n},
\end{align*}
so that 
\[\max_{0\le \tau\le t} |CT(\tau)x_n|\le \frac{2}{n} \rightarrow 0,\quad\text{as}\quad n\rightarrow \infty. \]

On the other hand, the isometry of $T(\cdot)$ yields for all $t\geq 0$ and  all $n\in\mathbb N$ that
\[\|T(t)x_n\|_X=\|x_n\|_X=1.\]

Assume that system \eqref{eq:example system} is eOSS. Then there exist $M,\mu>0$  and $\gamma\in \mathcal{K}$ such that for all $t\ge0$  and all \(n\in\mathbb N\),
\begin{align*}
1=\|T(t)x_n\|_X&\le M e^{-\mu t}+\gamma\left(\max_{0\le \tau\le t} |CT(\tau)x_n|\right)\le  Me^{-\mu t}+\gamma\left(\frac{2}{n}\right).
\end{align*}

 Since \(\gamma\in\mathcal K\), we have \(\gamma(2/n)\to 0\) as \(n\to\infty\).
Hence, for each fixed \(t\geq 0\), letting \(n\to\infty\) gives
\begin{align*} 1\leq Me^{-\mu t}. \end{align*}
 Now, letting \(t\to\infty\) yields a contradiction.  Thus, system \eqref{eq:example system} is not eOSS. 
\end{proof}

\subsection{Exponential OSS does not imply  the existence of  an eOSS Lyapunov function in dissipation form}

The next example shows that, in this generality, the converse eOSS Lyapunov theorem fails in dissipation form.
 This shows a fundamental difference between the dissipation form and the implication form in OSS Lyapunov theory.

Let $X=C_{2\pi}(\mathbb R,\C)$ be the space of all $2\pi$-periodic, continuous, complex valued functions on $\R$,  equipped with the supremum norm
$\|x\|_X:=\max_{s\in\mathbb R}|x(s)|.$ 
Let $Y=\mathbb C$.
Consider the system
\begin{equation}\label{eq:eoss-not-imply-lyap}
\dot x(t)=Ax(t), \quad y(t)=Cx(t),
\end{equation}
where \(A\) generates the \emph{left translation group}  $T(\cdot)$ on $X$ given for any $x \in X$ by  
\[
(T(t)x)(s)=x(s+t),
\quad s\in\mathbb R, \ t\in\R.
\]

The bounded output operator $C:X\to\mathbb C$ is defined by 
\[
Cx=x(0),\quad x\in X.
\]

Then $X$ is a Banach space and $T(\cdot)$ is a group of isometries on $X$; see  \cite[p.~34-35]{engel2000one}.

\begin{proposition}
\label{prop:no eOSS LF}
 The following hold:
\begin{itemize}
       \item[(i)] System \eqref{eq:eoss-not-imply-lyap} admits a coercive eOSS Lyapunov function in implication form with linear Lyapunov gain and $p=1$.       
    \item[(ii)] System \eqref{eq:eoss-not-imply-lyap} is eOSS with  linear gain.
    \item[(iii)] For any given $\rho\in \mathcal{P}$, $\sigma \in\mathcal{K}_\infty$, system \eqref{eq:eoss-not-imply-lyap} does not admit a  function $V: X\to \R_+ $, which is continuous along the trajectories (i.e., Definition~\ref{def:eOSS LF-dissipative}(ii) holds) and  satisfies the dissipation inequality
    \begin{align}\label{eq:general V}
         \dot{V}(x)\le -\rho(\|x\|_X)+\sigma(\|Cx\|_Y),\quad x\in X.
    \end{align}        
    \item[(iv)] There is no coercive eOSS Lyapunov function in dissipation form for system \eqref{eq:eoss-not-imply-lyap}.
\end{itemize}
\end{proposition}

\begin{proof}
(i). Let $\lambda>0$. Define
\begin{align*} V(x):= \max_{s\in[-2\pi,0]}e^{ \lambda s} |x(s)|,\quad x\in X.  \end{align*}
Note that $\max_{s\in[-2\pi,0]}|x(s)|=\|x\|_X$. Then 
\begin{align}\label{eq:piV}
e^{-2\pi\lambda}\|x\|_X\le V(x)\le \|x\|_X,\quad x\in X.
\end{align}
Thus $V$ is an equivalent norm on $X$, i.e., it is coercive with $p=1$. Moreover, since $V$ is a norm,  for all $x_1, x_2\in X$,
\begin{align*}
    |V(x_1)-V(x_2)|\le V(x_1-x_2) \le\|x_1-x_2\|_X.
\end{align*}
Therefore, $V$ is globally Lipschitz continuous.

Let $\chi(r):=3e^{2\pi\lambda }r$, $r\ge0$, which is a linear $\mathcal{K}_\infty$-function.  To verify the implication estimate \eqref{eq:aboss Implication}, fix an arbitrary $x\in X$ such that
\[
\|x\|_X\ge \chi(\|Cx\|_Y)=\chi(|x(0)|) = 3e^{2\pi\lambda} |x(0)|.
\]
Then 
\[
|x(0)|\le \frac13 e^{-2\pi\lambda}\|x\|_X\stackrel{\eqref{eq:piV}}{\le} \frac13V(x). 
\]
If $x=0$, \eqref{eq:aboss Implication} holds. Let $x\neq0.$ By the continuity of $x$, for all sufficiently small $\tau>0$,
\begin{align}\label{eq:small}
\max_{s\in[0,\tau]} e^{\lambda s} |x(s)|\le \frac12 V(x).
\end{align}

Fix $\tau \in(0,2\pi)$ such that \eqref{eq:small} holds. Note that $[-2\pi+\tau,0]\subset[-2\pi,0].$ By the definition of $T(\cdot)$, \eqref{eq:piV} and \eqref{eq:small}, we obtain the estimate
\begin{align*}
    V(T(\tau)x)= \max_{s\in [-2\pi,0]} e^{\lambda s}|(T(\tau) x)(&s)|
   = \max_{s\in [-2\pi,0]} e^{\lambda s} |x(s+\tau)|\\
   &= e^{-\lambda \tau} \max_{s\in[-2\pi+\tau,\tau]}e^{\lambda s}|x(s)|\\
    &= e^{-\lambda \tau} \max\Big\{\max_{s\in[-2\pi+\tau,0]}e^{\lambda s}|x(s)|, \max_{s\in[0,\tau]}e^{\lambda s}|x(s)|\Big\}\\
    &\le e^{-\lambda \tau} \max\Big\{V(x), \frac12V(x)\Big\}\\
    &= e^{-\lambda \tau} V(x).
\end{align*}

It follows that
\begin{align*} \dot{V}(x)\le-\lambda V(x)\stackrel{\eqref{eq:piV}}{\le} -\lambda e^{-2\pi \lambda} \|x\|_X.  \end{align*}

Therefore $V$ is a coercive eOSS Lyapunov function in implication form and satisfies the decay implication \eqref{eq:aboss Implication} with linear Lyapunov gain $\chi(r):=3e^{2\pi\lambda }r$, $r\ge0$, $p=1$, and $\tilde{\alpha} = \lambda e^{-2\pi \lambda}$. 

Statement (ii) follows from (i), in combination with Theorem~\ref{thm:direct eOSS} and Proposition~\ref{prop:equi OSS eOSS LG}.

(iii). Now we want to show that there is no function that is continuous along the trajectories and satisfies the dissipation inequality \eqref{eq:general V}.
Fix $ \rho\in\mathcal{P},  \sigma\in\mathcal K_\infty$ and assume, for contradiction, that there exists a function
$V$  for system \eqref{eq:eoss-not-imply-lyap} satisfying \eqref{eq:general V}.

Choose $\varepsilon\in(0,2\pi)$, and $x_\varepsilon\in X$  satisfying $x_\varepsilon(s)=0$ for $s\in[\varepsilon,2\pi]$, and
\[
0\le x_\varepsilon(s)\le 1 \quad \forall s\in\mathbb R,\ \|x_\varepsilon\|_X=1.
\]
Since $\sigma$ is strictly increasing and $0\le x_\varepsilon\le 1$, it follows that
\begin{align*}
0\le \sigma(|x_\varepsilon(s)|)\le \sigma(1)
\quad \forall s\in[0,2\pi]
\end{align*}
and so
\begin{align*}
\int_0^{2\pi} \sigma(|x_\varepsilon(s)|)\, \mathrm ds
= \int_0^{\varepsilon} \sigma(|x_\varepsilon(s)|)\, \mathrm ds \leq \sigma(1)\varepsilon.
\end{align*}
Choose $\varepsilon\in(0,2\pi)$ sufficiently small so that
$\sigma(1)\varepsilon<2\pi\rho(1).$  Define
\begin{align*}
f(t):=V(T(t)x_\varepsilon), \quad t\ge0.
\end{align*}
By assumption, $f$ is continuous. Moreover, since $T(2\pi)=\operatorname{Id}$, where the $\operatorname{Id}$ denotes the identity operator on $X$,  $f$ is $2\pi$-periodic 
\begin{align*}
f(t+2\pi)=V(T(t+2\pi)x_\varepsilon)=V(T(t)x_\varepsilon)=f(t), \quad t\geq 0.
\end{align*}

Next, for every $t\ge0$, applying \eqref{eq:general V} to the state $T(t)x_\varepsilon$, we obtain
\begin{align*}
D^+f(t)
=D^+ V(T(t)x_\varepsilon)&\le -\rho(\|T(t)x_\varepsilon\|_X)+\sigma(|CT(t)x_\varepsilon|)\\
&= -\rho(1)+\sigma(|x_\varepsilon(t)|) \, := g(t), 
\quad t\geq 0.
\end{align*}

 Then $g\in C([0,2\pi])$ and
$D^+f(t)\le g(t)$ for all $t\in[0,2\pi)$. By \cite[Proposition A.31]{mironchenko2023}, it follows that $f(2\pi)-f(0)\le \int_0^{2\pi} g(s)\, \mathrm ds$.
Consequently,
\begin{align*}
f(2\pi)-f(0)
&\le -2\pi\rho(1) +\int_0^{2\pi }\sigma(|x_\varepsilon(s)|)\, \mathrm ds\le -2\pi\rho (1) + \varepsilon \sigma(1)
<0,
\end{align*}
which contradicts $f(2\pi)=f(0)$. 

Item (iv) follows immediately from (iii).
\end{proof}

\subsection{The existence of an eOSS Lyapunov function does not imply exponential detectability}

The final example separates the existence of eOSS Lyapunov functions from exponential detectability.

Let  $X=\ell^2$ be the space of all complex-valued sequences
\(x=(x_n)_{n\in\mathbb N}\) such that
$
    \sum_{n=1}^{\infty} |x_n|^2 < \infty,
$
equipped with the norm
$
 \|x\|_{\ell^2}
    := (\sum_{n=1}^{\infty} |x_n|^2)^{1/2}
$ and let $Y=\ell^\infty$ be the space of all bounded complex-valued sequences with the norm $\|y\|_{\ell^\infty}: = \sup_{n\in\mathbb N} |y_n|.$

Choose a countable dense subset $(z_k)_{k\in\mathbb N}$ of the unit
sphere of $\ell^2$. Since
\begin{align*}
    \|x\|_{\ell^2}
    =
    \sup_{\|z\|_{\ell^2}=1}
    |\langle x,z\rangle_{\ell^2}|,
\end{align*}
and (for fixed $x\in \ell^2$) the map $z\mapsto \langle x,z\rangle_{\ell^2}$ is continuous, we have
\begin{align*}
    \|x\|_{\ell^2}
    =
    \sup_{k\in\mathbb N}
    |\langle x,z_k\rangle_{\ell^2}|,
    \quad x\in\ell^2 .
\end{align*}
Define
\begin{align*}
    C:\ell^2\to\ell^\infty,
    \quad
    Cx:=\bigl(\langle x,z_k\rangle_{\ell^2}\bigr)_{k\in\mathbb N}.
\end{align*}
Then $C$ is linear, and
\begin{align*}
    \|Cx\|_{\ell^\infty}
    =
    \sup_{k\in\mathbb N}
    |\langle x,z_k\rangle_{\ell^2}|
    =
    \|x\|_{\ell^2},
    \quad x\in\ell^2.
\end{align*}
Thus, $C$ is a linear isometry from $\ell^2$ into $\ell^\infty$. Equivalently, $C:\ell^2\to\operatorname{Im}C$ is an isometric isomorphism.

Consider the system
\begin{align}\label{sys:l_2}
    \dot x(t)=0,\quad y(t)=Cx(t).
\end{align}
Then $A=0$ and $T(t)=\operatorname{Id}$, $t\geq 0$.
\begin{proposition}\label{prop:not-dete}
The following hold:
\begin{enumerate}
    \item [(i)] System \eqref{sys:l_2} admits a coercive eOSS Lyapunov function in dissipation form with linear Lyapunov gain and $p=1$.
    \item [(ii)] System \eqref{sys:l_2} is not exponentially detectable.
   \item [(iii)] System \eqref{sys:l_2} is exactly observable on every finite time interval. More precisely, for every $\tau>0$, there exists $C_\tau>0$ such that 
   \begin{equation*}
      \int_0 ^\tau \|CT(t)x\|_Y^2 \,\mathrm{d}t\ge C_\tau\|x\|_X^2,\quad x\in X.
   \end{equation*}
\end{enumerate} 
\end{proposition}

\begin{proof}
(i). We are going to show that the system admits a coercive eOSS
Lyapunov function. Define
\[
V(x):=\|x\|_{\ell^2}, \quad x\in X.
\]
Then $V$ is coercive with $p=1$ and continuous. Since
$V(T(\tau)x)=V(x)$ for all $\tau\geq 0$, we have $\dot V(x)=0$.
Moreover,
\[
\|Cx\|_{\ell^\infty}=\|x\|_{\ell^2}.
\]
Thus, for any $\alpha>0$ and any $k>\alpha$,
\[
\dot V(x)=0
\leq -\alpha\|x\|_{\ell^2}+k\|Cx\|_{\ell^\infty}.
\]
Hence \eqref{sys:l_2} admits a coercive eOSS Lyapunov function in
dissipation form with linear Lyapunov gain.

(ii). Suppose, for contradiction, that \eqref{sys:l_2} is exponentially
detectable. Note that $A=0$ and
$C\in\mathcal L(\ell^2,\ell^\infty)$. Then there exists a bounded
operator $L\in\mathcal L(\ell^\infty,\ell^2)$ such that
$LC\in\mathcal L(\ell^2)$ generates an exponentially stable
$C_0$-semigroup on $\ell^2$. This implies that there exist constants
$M,\omega>0$ such that
\begin{align*}
\|e^{(LC)t}x\|_{\ell^2}
\leq Me^{-\omega t}\|x\|_{\ell^2},
\quad x\in\ell^2,\ t\geq0.
\end{align*}

Since $LC$ is a bounded operator and the semigroup generated by $LC$ is
exponentially stable, the spectrum of $LC$ is contained in
\begin{align*}
\{\lambda\in\mathbb C \mid \operatorname{Re}\lambda<0\}.
\end{align*}
Hence $0\in\rho(LC)$, where $\rho(LC)$ denotes the resolvent set of
$LC$. Thus $LC\in\mathcal L(\ell^2)$ is invertible with bounded inverse.

Define
\begin{align*}
P:=C(LC)^{-1}L:\ell^\infty\to\ell^\infty.
\end{align*}
Then $\operatorname{Im}P\subset\operatorname{Im}C$. Moreover, for all
$x\in\ell^2$,
\begin{align*}
P(Cx)
=C(LC)^{-1}LCx
=Cx.
\end{align*}
Hence $P$ is the identity on $\operatorname{Im}C$. Therefore,
$\operatorname{Im}C\subset\operatorname{Im}P$, and consequently
\begin{align*}
\operatorname{Im}P=\operatorname{Im}C.
\end{align*}

Furthermore, for every $y\in\ell^\infty$, we have
$Py\in\operatorname{Im}C$. Since $P$ is the identity on
$\operatorname{Im}C$, it follows that
\begin{equation*}
P^2y=P(Py)=Py.
\end{equation*}
Thus $P^2=P$, and $P$ is a bounded projection from $\ell^\infty$ onto
$\operatorname{Im}C$. This implies that $\operatorname{Im}C$ is
complemented in $\ell^\infty$.

Since $C:\ell^2\to\operatorname{Im}C$ is an isometric isomorphism,
$\operatorname{Im}C$ is infinite-dimensional and separable. However,
by \cite{lindenstrauss1967complemente}, every infinite-dimensional
complemented subspace of $\ell^\infty$ is isomorphic to $\ell^\infty$.
This is impossible, since $\ell^\infty$ is not separable, whereas
$\operatorname{Im}C$ is separable. Therefore \eqref{sys:l_2} is not
exponentially detectable.

(iii). Since $T(t)=\operatorname{Id}$ and $C$ is an isometry, for every $x\in \ell^2$ and $\tau>0,$ 
\begin{equation*}
    \int_0^\tau \|CT(t)x\|_{\ell^\infty}^2\,\mathrm{d}t =\int_0^\tau \|Cx\|_{\ell^\infty}^2\,\mathrm{d}t=\tau \|x\|_{\ell^2}^2,
\end{equation*}
and therefore the system is exactly observable on every finite time interval $[0,\tau].$
\end{proof}

\begin{remark}
The implication from the existence of a coercive eOSS Lyapunov function in dissipation form with linear Lyapunov gain to exponential detectability fails even for the Hilbert
state space $X=\ell^2$ and the bounded generator $A=0$. However, the example
\eqref{sys:l_2} still relies on the fact that the output space $Y$ is not Hilbert.

For the trivial dynamics $A=0$, this type
of construction cannot arise as soon as the output space $Y$ is Hilbert, with $X$ merely assumed to be a Banach space. To see this, suppose that the system admits a coercive eOSS Lyapunov function in dissipation form
with arbitrary $\alpha, p>0$ and arbitrary $\sigma\in\mathcal K_\infty$, that is, 
\begin{align*}
0=\dot{V}(x)\leq-\alpha\|x\|_X^p+\sigma(\|Cx\|_{Y}).
\end{align*}
Taking $\|x\|_X=1$, we obtain
\[
    \alpha \le \sigma(\|Cx\|_Y).
\]
Since $\sigma\in\mathcal K_\infty$, it follows that
\[
    \|Cx\|_Y\ge \sigma^{-1}(\alpha)>0,
    \quad \|x\|_X=1 .
\]
For $x\neq0$, it yields
\[
    \|Cx\|_Y=\|x\|_X\bigr\|C \frac{x}{\|x\|_X}\bigr\|_Y\ge \sigma^{-1}(\alpha)\|x\|_X,
\]
which implies $ \|Cx\|_Y\ge \sigma^{-1}(\alpha)\|x\|_X $ for all $x\in X$.

Thus $C(X)$ is closed in $Y$  and
$C^{-1}:C(X)\to X$ is bounded. Since $Y$ is Hilbert, there is
an orthogonal projection $P_{C(X)}:Y\to C(X)$. Hence, for any $\lambda>0$,
the operator
\begin{align*} 
    L=-\lambda C^{-1}P_{C(X)}
 \end{align*}
is bounded from $Y$ to $X$ and satisfies $LC=-\lambda \operatorname{Id}$. Therefore
$A+LC=-\lambda \operatorname{Id}$ generates an exponentially stable semigroup.
\end{remark}

\section{Observers and spectral decompositions}\label{sec: decomposition}

Section~\ref{sec:counterexam} shows that exponential zero-detectability does not imply eOSS or exponential detectability in general. In this section, we further investigate the relations between these properties. In Section~\ref{sec:Exponential detectability and robust Luenberger observers}, we provide an observer-based interpretation of the implication from exponential detectability to eOSS. This is based on the standard Luenberger observer construction for exponentially detectable infinite-dimensional systems; see \cite{Curtain2020}. 

In Section~\ref{sec:Spectrum decomposition at 0 and exponential detectability}, we study conditions under which exponential zero-detectability implies exponential detectability.
To this end, we consider systems for which the spectral subspace associated with the spectrum in the closed right half-plane is finite-dimensional and the remaining part is exponentially stable. A related result is given in \cite[Theorem~8.1.7]{Curtain2020}, where exponential detectability is characterized by a spectrum decomposition at zero with an exponentially stable part and a finite-dimensional observable subsystem, under the additional assumption that the output operator has finite rank. Here, no finite-rank assumption is imposed on the original output operator $C$.

Under the spectral decomposition assumption, exponential zero-detectability and exponential detectability are equivalent. Together with the implications established in Theorem~\ref{thm:single direction}, this gives equivalent characterizations of eOSS in terms of exponential detectability, the corresponding Lyapunov conditions, and exponential zero-detectability.

\subsection{Exponential detectability and robust Luenberger observers}
\label{sec:Exponential detectability and robust Luenberger observers}

Proposition~\ref{prop:coer eOSS LF} derives eOSS from exponential detectability by constructing
an eOSS Lyapunov function. We present the same implication from the perspective of Luenberger observers. In this subsection, we show that exponential detectability allows us to construct an exponentially robust Luenberger observer, whose existence implies eOSS. Exponential detectability is also used to construct a Luenberger observer, which further implies eOSS.

Following \cite[Definition 8.3.1]{Curtain2020} we consider, for a bounded output injection operator $L\in\mathcal L(Y,X)$, the observer driven by the perturbed output $y(\cdot)+d(\cdot)$
\begin{align}\label{sys:robust}
\dot{\hat{x}}(t)
&=
A\hat{x}(t)+L\bigl(C\hat{x}(t)-(y(t)+d(t))\bigr)
\end{align}
with initial condition $\hat{x}(0)=\hat{x}_0 \in X$. Here
$d\in L^\infty_{\, \mathrm{loc}}([0,\infty),Y)$ denotes an unknown output disturbance.
For $t>0$, we use the notation
\[
\|d\|_{L^\infty([0,t),Y)}
:=
\operatorname*{ess\,sup}_{s\in(0,t)}\|d(s)\|_Y
\] 
and for $t=0$ we set this quantity equal to zero.

Let $x(0)=x_0$.
Given $x_0,\hat{x}_{0}\in X$, let $x(t)=\phi(t,x_0)$ denote the mild solution of system \eqref{eq:linear-system} with $x_0$, and $\hat{x}(t)$ denote the mild solution of the system \eqref{sys:robust} with $\hat{x}_{0}$. Define  the observation  error by
\begin{align*} e(t):=\hat{x}(t)-x(t), \quad t\geq 0. \end{align*}

Following the robust observer viewpoint of \cite[Definition 20]{sontag1997output}, we define
\begin{definition}
\label{def:luenberger-observer}
 The system \eqref{eq:linear-system}
 is said to  admit an \emph{exponentially robust Luenberger observer} if there exist $L\in\mathcal{L}(Y,X)$, $M,\mu>0$ and $\gamma\in \mathcal K$ such that the mild solutions of \eqref{sys:robust} satisfy for all $x_0,\hat{x}_{0}\in X$, all $t\ge0$ and all $d\in L^\infty_{\, \mathrm{loc}}([0,\infty),Y)$
\begin{align}\label{eq:robust}
   \|e(t)\|_X\le M e^{-\mu t}\|\hat{x}_{0}-x_0\|_X
    + \gamma (\|d\|_{L^\infty([0,t),Y)}).
\end{align}
 
\end{definition}

\begin{proposition}
Consider the following statements:
\begin{enumerate}
    \item[(i)] System~\eqref{eq:linear-system} is exponentially detectable.
    \item[(ii)] System \eqref{eq:linear-system} admits an exponentially robust Luenberger observer.
    \item[(iii)] System \eqref{eq:linear-system} is eOSS.
\end{enumerate}
Then
\[
\text{(i) $\Rightarrow$  (ii) $\Rightarrow$ (iii)}.
\]
\end{proposition}

\begin{proof}
(i) $\Rightarrow$ (ii). This implication is based on the standard Luenberger observer construction; see \cite[Lemma 8.3.2]{Curtain2020} for the infinite-dimensional setting.  Since the system is exponentially detectable, there exists an output injection operator $L\in\mathcal L(Y,X)$ such that  $A+LC$ generates an exponentially stable semigroup $ T_{LC}(\cdot)$. I.e., there exist constants 
$M\ge 0$ and $\mu>0$ such that
\[
   \|T_{LC}(t)\| \le M e^{-\mu t}, \quad t\ge 0.
\]

For this choice of $L$, consider \eqref{sys:robust}. Since $y(t)=Cx(t)$, a standard manipulation of the integral representation of mild solutions shows that the error satisfies
    \begin{align*}
    e(t)
    =T_{LC}(t)e(0)-
    \int_0^t T_{LC}(t-s)Ld(s)\, \mathrm ds.
\end{align*}
    
Using the exponential stability of $T_{LC}(\cdot)$, we get
\begin{align}\label{eq:robust-LG}
    \|e(t)\|_X
    \leq\;&
    M e^{-\mu t}\|e(0)\|_X+
    M\|L\|_{\mathcal L(Y,X)}
    \int_0^t e^{-\mu(t-s)}
    \|d(s)\|_Y\, \mathrm ds  \notag\\
    \leq\;&
    M e^{-\mu t}\|\hat{x}_{0}-x_0\|_X
    +
    \frac{M\|L\|_{\mathcal L(Y,X)}}{\mu}\|d\|_{L^\infty([0,t),Y)}.
\end{align}  
Therefore, if $L\neq0$, \eqref{eq:robust} holds with $\gamma(r)=\frac{M\|L\|_{\mathcal L(Y,X)}}{\mu}r$. If $L=0$, then \eqref{eq:robust} holds with any $\gamma\in\mathcal{K}.$  \eqref{sys:robust} is an exponentially robust Luenberger observer for \eqref{eq:linear-system}.

(ii) $\Rightarrow$ (iii). Appealing to Definition~\ref{def:luenberger-observer}, fix an observer and a gain $\gamma$ as in \eqref{eq:robust}. Choose $\hat{x}_0=0$ and let $x_0\in X$ be arbitrary. Define $y(t) := C T(t)x_0$, $t\geq 0$, and
\begin{align*} 
    d(t)=-y(t), \quad t \geq 0.
 \end{align*}
Since $y(\cdot)$ is continuous, this choice is admissible as an element of
$L^\infty_{\, \mathrm{loc}}([0,\infty),Y)$, and
\[
    \|d\|_{L^\infty([0,t),Y)}
    =
   \max\limits_{s\in[0,t]}\{\|y(s,x_0)\|_Y\}, \quad t>0.
\]
As $y(t)+d(t) = 0$,  \eqref{sys:robust} reduces to
\[
    \dot{\hat{x}}(t)
    =
    (A+LC)\hat{x}(t),
    \quad
    \hat{x}(0)=0,
\]
with unique mild solution 
  $
    \hat{x}(t)\equiv 0.
$
Thus
\[
    e(t)=\hat{x}(t)-x(t)=-x(t), \quad t\geq 0.
\]
By assumption,  \eqref{eq:robust} is satisfied and we obtain for $t>0$
\[
    \|x(t)\|_X
    \le
    M e^{-\mu t}\|x_0\|_X
    +
  \gamma( \max\limits_{s\in[0,t]}\|y(s,x_0)\|_Y).
\]
The case $t=0$ is immediate. Hence the system is eOSS.
\end{proof}

\subsection{Spectrum decomposition at \texorpdfstring{$0$}{0} and exponential detectability}
\label{sec:Spectrum decomposition at 0 and exponential detectability}

The spectral decomposition separates the finite-dimensional spectral subspace associated with the spectrum in the closed right half-plane from an exponentially stable complement, see \cite[Definition 8.1.5, Theorem 8.1.6]{Curtain2020}. This structure allows the detectability analysis to be reduced to a finite-dimensional subsystem.

We first recall that a subspace $V\subset X$ is called
\emph{$T(\cdot)$-invariant} if
\[
T(t)V\subset V \quad \forall t\ge0.
\]
Here and below, $X=X_u\oplus X_s$ means that $X$ is an internal direct sum of  $X_u$ and $X_s$, which are closed subspaces of $X$, namely: $X_u\cap X_s=\{0\}$, and every $x\in X$ admits a unique decomposition
$x=x_u+x_s$ with $x_u\in X_u$ and $x_s\in X_s$.

\begin{assumption}\label{ass:decomp}
There are closed, $T(\cdot)$-invariant subspaces $X_u,X_s\subset X$ so that
\[
X = X_u \oplus X_s, \quad \dim X_u < \infty,
\]
where $X_u$ is the spectral subspace of $A$ associated with the spectrum in the closed right half-plane. Moreover, there exist constants $M_s,\omega_s>0$ such that
\begin{equation}\label{eq:Ts-stable}
\|T(t)x_s\|_X\le M_s e^{-\omega_s t}\|x_s\|_X
\quad \forall x_s\in X_s,\ \forall t\ge0.
\end{equation}
\end{assumption}

Such decompositions, which separate finitely many unstable modes from an exponentially stable infinite-dimensional part, are classical in PDE control;
see \cite{russell1978controllability}.
For  sufficient conditions ensuring this assumption, we refer to \cite[Section 2.4]{Curtain2020}.

In what follows, we write
\[
T_u(t):=T(t)|_{X_u},\quad T_s(t):=T(t)|_{X_s},
\]
and denote their generators by $A_u$ and $A_s$, respectively. The corresponding
output operators are denoted by
\[
C_u:=C|_{X_u}, \quad C_s:=C|_{X_s}.
\]

\begin{theorem}\label{thm:assum-eoss}
Let Assumption~\ref{ass:decomp} hold.
 The following statements are equivalent:
\begin{enumerate}
    \item[(i)] System \eqref{eq:linear-system} is exponentially detectable.
    \item[(ii)] System \eqref{eq:linear-system} is eOSS.
     \item[(iii)] System \eqref{eq:linear-system} is exponentially zero-detectable.
    \item[(iv)] The pair $(A_u,C_u)$ is exponentially zero-detectable.
 \item[(v)] The pair $(A_u,C_u)$ is exponentially detectable.
  \item[(vi)] The pair $(A_u,C_u)$ is observable.
\end{enumerate} 
In particular, all the statements (i)-(ix) in Theorem~\ref{thm:single direction} are equivalent.
\end{theorem}
\begin{proof}
(i) $\Rightarrow$ (ii) $\Rightarrow$ (iii). This follows from Theorem~\ref{thm:single direction}.

(iii) $\Rightarrow$ (iv). For any $x_u\in X_u$, the $T(\cdot)$-invariance of $X_u$ implies
\[
T(t)x_u=T_u(t)x_u\in X_u,\quad t\geq0.
\]

Hence, if
$
C_uT_u(t)x_u=0$ for all $ t\geq0$, then
$
CT(t)x_u=0 $ for all $t\geq0.$
Since system~\eqref{eq:linear-system} is exponentially
zero-detectable, there exist $M,\omega>0$ such that
\[
\|T_u(t)x_u\|_X
=
\|T(t)x_u\|_X
\leq
Me^{-\omega t}\|x_u\|_X,
\qquad t\geq0.
\]
Therefore, $(A_u,C_u)$ is exponentially zero-detectable.

(iv) $\Leftrightarrow$ (v) $\Leftrightarrow$ (vi). By \cite[p. 317]{sontag2013mathematical}, exponential zero-detectability and exponential
detectability of the finite-dimensional pair $(A_u,C_u)$ are equivalent. Since
the spectrum of $A_u$ is contained in the closed right half-plane, no nonzero trajectory can have an exponentially decaying trajectory. Therefore, if $(A_u,C_u)$ is exponentially detectable, any state producing identically zero output must be zero. It is well-known that this is exactly observability.

(vi) $\Rightarrow$ (i).  To apply \cite[Theorem~8.1.7]{Curtain2020} to the finite-dimensional unstable part,
let
$Y_u:=C(X_u).$
Since $X_u$ is finite-dimensional, $Y_u$ is a finite-dimensional
closed complemented subspace of $Y$. Hence, there exists a bounded projection
$Q_u\in\mathcal L(Y,Y_u)$
such that
$Q_uy=y$ for all $y\in Y_u$.

Regard $C_u$ as an operator from $X_u$ to $Y_u$. Applying
\cite[Theorem~8.1.7]{Curtain2020} to $(A_u,C_u)$, we obtain an operator
$L_u\in\mathcal L(Y_u,X_u)$ such that $A_u+L_uC_u$ is exponentially stable.
Define
\[
Ly:=L_uQ_uy\in X_u\subset X, \quad y\in Y.
\]
With respect to the decomposition $X=X_u\oplus X_s$, we have
\[
A+LC=
\begin{pmatrix}
A_u+L_uC_u & L_uQ_uC_s\\
0 & A_s
\end{pmatrix}.
\]
Since both diagonal blocks generate exponentially stable semigroups, the upper triangular coupling term is bounded, and hence  $A+LC$ generates an exponentially stable $C_0$-semigroup on $X$.
 Therefore, system~\eqref{eq:linear-system}
is exponentially detectable.
\end{proof}

\section{Exponential output-to-state stability of a parabolic system}\label{sec: example}

In this section, the abstract results developed earlier are applied to a concrete parabolic PDE system. An explicit coercive UES Lyapunov function is first constructed to verify UES under suitable conditions. Next, the finite-dimensional detectability result from Theorem~\ref{thm:assum-eoss} is used to show that the system is eOSS with linear gain for all parameter values. Since Assumption~\ref{ass:decomp} is satisfied, the parabolic system admits an eOSS Lyapunov function and is exponentially detectable.

Let $a>0$, and $c, q\in \mathbb{R}$. In this section $L^2(0,1)$ and the Sobolev spaces are real spaces.
Consider the linear parabolic equation with output
\begin{subequations}\label{expde}
\begin{align}\label{sys:sub exam}
x_{t}(z,t) &=ax_{zz}(z,t)+c x(z,t),  \\
\label{eq:y1}
y(t)  &=\int_0^1 x(z,t)\, \mathrm dz,
\end{align}
\end{subequations}
where  $x_t(z,t)$ and $ x_{zz}(z,t)$ denote, respectively,  the first partial derivative of $x(z,t)$ with respect to $t$ and the second partial derivative of $x(z,t)$ with respect to $ z$.
The equation is defined for $z\in(0,1), t>0$, and is endowed  with  boundary conditions
(namely, a Robin condition on the left end and a Dirichlet condition at the right end)
\begin{align}\label{boundary}
x_z(0,t)=-qx(0,t),\quad 
x(1,t)=0,
\end{align}
and an initial condition
\begin{equation*}    
x(z,0)=x_0(z).
\end{equation*}
Define
\begin{center}
$A =a\frac{\, \mathrm d^2}{\, \mathrm dz^2}+c$ on $X=L^2(0,1)$     
\end{center}
with domain
\begin{equation}    
D(A) = \{ x \in H^2(0,1) \mid  x_z(0) = - q x(0),\ x(1)=0 \},
\end{equation}
where $H^2(0,1)$ is the standard Sobolev space of square integrable functions with (weak) first and second derivatives that are again square integrable.

According to \cite[Exercise VI.4.7]{engel2000one}, $A$ generates an analytic semigroup.

System \eqref{expde} with boundary conditions \eqref{boundary} is unstable if $c$ or $q$ are positive and sufficiently large, see \cite[p. 13]{smyshlyaev2010adaptive}. We are going to  show that in some cases, the system is UES  and a  Lyapunov function  is easy to find.
\begin{proposition}\label{prop:UGASLF}
Let  $X=L^2(0,1)$ and $Y=\mathbb{R}$.
Assume that $c<\frac{a\pi^2}{4}$ and $q<\min\{1, \frac{a\pi^2-4c}{a(\pi^2+4)}\}$. Then
\begin{equation}\label{eq:LF-example}  
V(x):=\frac{1}{2}\int_0^1x^2(z)\, \mathrm dz=\frac{1}{2}\|x\|_X^2  
\end{equation}
is a coercive UES Lyapunov function for \eqref{expde} with \eqref{boundary}, and thus \eqref{expde} with \eqref{boundary} is UES. Consequently, \eqref{expde} with \eqref{boundary} is eOSS with linear gain.
\end{proposition}

\begin{proof}
Consider the Lyapunov function candidate from \eqref{eq:LF-example},
which is obviously coercive.
Assume that $x\in D(A)$.  Using integration by parts and the boundary conditions \eqref{boundary} yields
\begin{align}\label{deri}
\dot{V}(x)= \int_0^1x(z)x_t(z)\, \mathrm dz\notag
&=\int_0^1x(z)(ax_{zz}(z)+c x(z))\, \mathrm dz\notag\\
&=-ax(0)x_z(0)-\int_0^1ax_z^2(z)\, \mathrm dz+\int_0^1cx^2(z)\, \mathrm dz\notag\\
&=aqx^2(0)-a\int_0^1x_z^2(z)\, \mathrm dz+c\int_0^1x^2(z)\, \mathrm dz.
\end{align}

By a variation of Wirtinger's inequality \cite[Remark B.1]{smyshlyaev2010adaptive} and $x(1)=0$,
we have
\begin{equation*}
 -\int_0^1x_z^2(z)\, \mathrm dz\leq-\frac{\pi^2}{4}\int_0^1x^2(z)\, \mathrm dz.
\end{equation*}
For $q\leq0$, we have
\begin{align*}
 \dot{V}(x)\leq -a\int_0^1x_z^2(z)\, \mathrm dz+c\int_0^1x^2(z)\, \mathrm dz
 \leq&-\frac{a\pi^2}{4}\int_0^1x^2(z)\, \mathrm dz+c\int_0^1x^2(z)\, \mathrm dz\\
\leq & -\left(\frac{a\pi^2}{4}-c\right)\int_0^1x^2(z)\, \mathrm dz.
\end{align*}

For $0<q<1$, 
applying Agmon's inequality \cite[Lemma B.2]{smyshlyaev2010adaptive} and using $x(1)=0$, we have
\begin{align*}
\max_{z \in [0,1]} |x(z)|^2 &\leq
2\left(\int_0^1 x^2(z)\, \mathrm dz\right)^{\frac{1}{2}}\left(\int_0^1 x_z^2(z)\, \mathrm dz\right)^{\frac{1}{2}}\leq \int_0^1 x^2(z)\, \mathrm dz+\int_0^1 x_z^2(z)\, \mathrm dz,
\end{align*}
and \eqref{deri} becomes
\begin{align*}
\dot{V}(x)\leq & \, aq \int_0^1 x^2(z)\, \mathrm dz+aq\int_0^1 x_z^2(z)\, \mathrm dz
-a\int_0^1x_z^2(z)\, \mathrm dz+c\int_0^1x^2(z)\, \mathrm dz\notag\\
=& \, (aq+c) \int_0^1 x^2(z)\, \mathrm dz
-a(1-q)\int_0^1 x_z^2(z)\, \mathrm dz\notag\\
\leq&\, -\frac{a\pi^2(1-q)}{4}\int_0^1 x^2(z)\, \mathrm dz+(aq+c)\int_0^1 x^2(z)\, \mathrm dz\\
\leq& \,-\left(\frac{a\pi^2(1-q)-4(aq+c)}{4}\right)\int_0^1 x^2(z)\, \mathrm dz.
\end{align*}

Since $D(A)$ is invariant under the semigroup,
if $c<\frac{a\pi^2}{4}$ and
$q<\min\{1, \frac{a\pi^2-4c}{a(\pi^2+4)}\}$,
then the previous estimates show that
there is $\alpha>0$ such that
\[
D^+V(\phi(t,x))\le -\alpha V(\phi(t,x)),\quad x\in D(A).
\]
By \cite[Proposition A.35]{mironchenko2023}, we get
\[
V(\phi(t,x))\le e^{-\alpha t}V(x),
\quad x\in D(A).
\]

As $D(A)$ is dense in $X$,
the estimate extends to all $x\in X$ by the boundedness of $T(\cdot)$
for each fixed $t\ge0$ and the continuity of $V$ on $X$:
\begin{align}\label{eq:e-alpha}
V(T(t)x)\le e^{-\alpha t}V(x),\quad x\in X,\ t\ge0.
\end{align}

Since $V(x)=\frac12\|x\|_X^2$, we obtain
\[
    \|T(t)x\|_X\le e^{-\frac{\alpha}{2}t}\|x\|_X,
    \quad x\in X,\ t\ge0.
\]
Moreover, the  estimate  \eqref{eq:e-alpha} implies 
\begin{align*} \dot{V}(x)\le -\alpha V(x)= -\frac{\alpha}{2}\|x\|^2_X,\quad x\in X. \end{align*}
Thus $V$ is a coercive UES Lyapunov function on $X$. In particular, 
\eqref{expde} with \eqref{boundary} is UES. Consequently, it is eOSS with 
linear gain.

\end{proof}

\begin{remark}\label{rem: different p}
 The coercive UES Lyapunov function in Proposition~\ref{prop:UGASLF} satisfies a dissipation estimate that is independent of the output. Therefore, under the same conditions as in Proposition~\ref{prop:UGASLF}, the system is UES and eOSS with linear gain for any nonlinear output of the form $ y_p(t) = \int_0^1 |x(z,t)|^p\, \mathrm dz$, where $0< p \le 2$.
\end{remark}

Next, we are going to show that the system \eqref{expde} with  \eqref{boundary} is eOSS for all parameter values using a decomposition.
\begin{proposition}\label{exproposition}
Let  $X=L^2(0,1)$ and $Y=\mathbb{R}$. 
System  \eqref{expde} with boundary conditions \eqref{boundary} is eOSS with linear gain for any $a>0$ and any $c,q\in\mathbb{R}$. 
\end{proposition}
\begin{proof}
We verify the assumptions of Theorem \ref{thm:assum-eoss}. 
Recall that the operator $A$ introduced above generates an analytic $C_0$-semigroup $T(\cdot)$ on $X=L^2(0,1)$. The operator $A$ is self-adjoint on $L^2(0,1)$.
Since $(0,1)$ is bounded, the embedding $H^2(0,1)\hookrightarrow L^2(0,1)$ is compact, and hence $A$ has compact resolvent (see \cite[Proposition~II.4.25]{engel2000one}). 

Furthermore, its spectrum consists  of eigenvalues $(\lambda_k)_{k\in\mathbb N}$ of $A$ with 
\begin{align*} -\infty<\cdots<\lambda_k<\cdots<\lambda_1\quad\text{and}\quad
\lambda_k \xrightarrow{k\to+\infty} -\infty, \end{align*} 
and the corresponding eigenfunctions $(\varphi_k)_{k\in\mathbb N}\subset D(A)$ 
form an orthonormal  basis of $L^2(0,1)$. 
In particular, every solution admits the expansion
\begin{align*}
& x_0=\sum_{k=1}^\infty \alpha_k\varphi_k,\quad \alpha_k = \langle x_0, \varphi_k \rangle_{L^2(0,1)}\\   
&x(\cdot,t)=T(t)x_0=\sum_{k=1}^\infty \alpha_k e^{\lambda_k t}\varphi_k(\cdot),
\end{align*}
where the series converges  in $L^2(0,1)$ for each $t \geq 0$.

For the output, we define
\begin{align*} 
y(t):=Cx(t)=\int_0^1 x(z,t)\, \mathrm dz
      =\sum_{k=1}^\infty \alpha_k e^{\lambda_k t} c_k, \end{align*}
where
\begin{align}
\label{eq:c_k_coefs_example}
 c_k := C\varphi_k = \int_0^1 \varphi_k(z)\, \mathrm dz.   
\end{align}

Since $\lambda_k\to -\infty$, the set
\begin{align*} 
I_u := \{k\in\mathbb N \mid  \lambda_k\ge 0\}
 \end{align*}
is finite. Note that if $I_u$ is empty, then the system is exponentially stable and the eOSS inequality holds trivially. Thus, we focus on the case when $I_u$ is nonempty.

We define the finite-dimensional subspace
\[
X_u := \operatorname{span}\{\varphi_k\mid  k\in I_u\},
\]
and its orthogonal complement $X_s := X_u^\perp$.
Then $X = X_u \oplus X_s$, where both $X_u$ and $X_s$ are closed and invariant under
the semigroup $T(\cdot)$. Accordingly, every solution admits the decomposition
\[
x(t) = x_u(t) + x_s(t), \quad x_u(t)\in X_u,\; x_s(t)\in X_s .
\]

Let $x_s(0)\in X_s$. Then
\[
x_s(t)=\sum_{k\notin I_u}\alpha_k e^{\lambda_k t}\varphi_k .
\]
By orthonormality,
\[
\|x_s(t)\|_{L^2(0,1)}^2
=\sum_{k\notin I_u}|\alpha_k|^2 e^{2\lambda_k t}.
\]
Since $I_u$ is finite and $\lambda_k\to-\infty$, there exists $\omega>0$ such that
$\lambda_k\le -\omega$ for all $k\notin I_u$. Hence
\begin{align*}
\|x_s(t)\|_{L^2(0,1)} \le e^{-\omega t}\|x_s(0)\|_{L^2(0,1)},\quad t\ge0,
\end{align*}
which implies \eqref{eq:Ts-stable} holds.

We now show that $c_k\neq 0$ for all $k\in I_u$. 
Consider the eigenvalue problem
\begin{align*}
a\varphi_k''(z)+(c-\lambda_k)\varphi_k(z)=0
\end{align*}
with boundary conditions
\begin{align*}
\varphi_k'(0)=-q\varphi_k(0),\quad\varphi_k(1)=0.
\end{align*}

By a slight modification of the standard computation in \cite[Chapter 4]{strauss2007partial},
we obtain the following eigenfunctions.

    \noindent\textbf{Case 1.} $c-\lambda_k>0$. Let $\mu_k := \sqrt{(c-\lambda_k)/a}>0$.
The eigenfunctions take the form
\begin{align*}
  \varphi_k(z) = \varphi_k(0)\bigl(\cos(\mu_k z) - \frac{q}{\mu_k}\sin(\mu_k z)\bigr),
\end{align*}
where  $\mu_k$  satisfies the  equation
 \begin{align}\label{eq:case1}
 \mu_k\cos\mu_k - q\sin\mu_k = 0,
 \end{align}
which is derived from the boundary condition $\varphi_k(1)=0$.
    
    \noindent\textbf{Case 2.} $c-\lambda_k=0$. A nontrivial eigenfunction exists only if $q=1$,
    in which case $\lambda_k=c$ and
    \begin{align*}
      \varphi_k(z)=\varphi_k(0)(1-z).
    \end{align*}

    \noindent\textbf{Case 3.} $c-\lambda_k<0$. Let $\kappa_k := \sqrt{(\lambda_k-c)/a}>0$.
    The eigenfunctions can be written as
    \begin{align*}
      \varphi_k(z) = \varphi_k(0)\bigl( \cosh(\kappa_k z) - \frac{q}{\kappa_k}\sinh(\kappa_k z)\bigr),
    \end{align*}
    where $\kappa_k$ is determined by  the  equation
     \begin{align*}
      \kappa_k\cosh\kappa_k - q\sinh\kappa_k = 0.
   \end{align*}

Plugging each of the expressions from cases~1–3 into \eqref{eq:c_k_coefs_example}, we obtain 
\begin{align*}
  c_k =
  \begin{cases}
    \displaystyle \frac{1-\cos\mu_k}{\mu_k\sin\mu_k}\varphi_k(0), & \text{if } c-\lambda_k>0,\\[1.2ex]
    \displaystyle \frac{1}{2}\varphi_k(0), & \text{if } c-\lambda_k=0 \text{ and } q=1,\\[1.2ex]
    \displaystyle \frac{\cosh\kappa_k-1}{\kappa_k\sinh\kappa_k}\varphi_k(0),
      & \text{if } c-\lambda_k<0.
  \end{cases}
\end{align*}

Moreover, $\varphi_k(0)\neq 0$. Indeed, if $\varphi_k(0)=0$, then the
boundary condition $\varphi_k'(0)=-q\varphi_k(0)$ gives
$\varphi_k'(0)=0$. By uniqueness for the corresponding second-order ODE,
this would imply $\varphi_k\equiv 0$, a contradiction.
 In case~1, the  equation \eqref{eq:case1}
implies $\sin\mu_k\neq 0$ and $\cos\mu_k\neq 1$.
In case~3, since $\kappa_k>0$, we know $\sinh\kappa_k>0$ and $\cosh\kappa_k>1$.
Therefore $c_k\neq 0$ for all $k\in I_u$.

The spectrum of $A_u$ is $\{\lambda_k\mid k\in I_u\}$. Since this
one-dimensional Sturm--Liouville problem has simple eigenvalues, each
eigenspace corresponding to $\lambda_k$ is spanned by $\varphi_k$. Moreover,
we have shown that
\[
C\varphi_k=c_k\neq0,\quad k\in I_u.
\]
Hence, no eigenvector of $A_u$ corresponding to an eigenvalue with a nonnegative real part belongs to $\ker C_u$. By the finite-dimensional
Hautus criterion, the pair $(A_u,C_u)$ is detectable. The proof is complete
via Theorem~\ref{thm:assum-eoss}.
\end{proof}

\begin{remark}
In particular, \eqref{expde} with \eqref{boundary}   admits a coercive eOSS Lyapunov function,  and is both exponentially detectable and exponentially zero-detectable.
\end{remark}

\section{Conclusion}\label{concl}

In this paper, we have developed a systematic theory for (exponential)  output-to-state stability of linear infinite-dimensional systems with bounded output operators.

We have shown that exponential OSS is natural in the sense that it can be characterized in terms of (suitably defined) OSS Lyapunov functions in implication form.
Also, eOSS of a system constitutes a necessary condition for the existence of a robust observer for the system. 

We proved that exponential detectability implies eOSS, which in turn implies exponential zero-detectability. 
The converse implications were shown to fail in general. Counterexamples
demonstrate that exponential zero-detectability does not imply eOSS, that
eOSS does not imply the existence of a coercive eOSS Lyapunov function in  dissipation form, and that such a Lyapunov function does not imply exponential detectability.  

The fact that the existence of an OSS Lyapunov function in a dissipation form and the existence of an OSS Lyapunov function in an implication form are not equivalent (already for linear systems with bounded output operators!) is rather surprising, and indicates that implication-form Lyapunov functions are, after all, more suitable for the analysis of this property. 

Note that OSS can be formulated and studied for nonlinear systems, see \cite{sontag1997output} for the ODE case and our preliminary results \cite{chen2025lyapunov} in the infinite-dimensional setting. As exponential detectability cannot be extended in a straightforward way to nonlinear systems, and exponential zero-detectability is a far too weak property, we again feel that OSS is a natural concept to describe nonlinear detectability.

On the other hand, when the unstable subspace is finite-dimensional, exponential zero-detectability
implies exponential detectability, and the corresponding eOSS characterizations become equivalent.

\bibliographystyle{siamplain}
\bibliography{AM-Sandbox/References,AM-Sandbox/Mir_LitList_NoMir,AM-Sandbox/MyPublications,fabian}
\end{document}